\documentclass[10pt, journal]{IEEEtran}

\usepackage{graphics} 
\usepackage{amsmath} 
\usepackage{amssymb}  
\usepackage{amsthm}
\usepackage{cite}
\usepackage{bm}
\usepackage{acronym}
\usepackage{paralist}
\usepackage{float}
\usepackage{color}
\usepackage{epstopdf}
\usepackage{multicol}
\usepackage{tikz}
\usepackage{hyperref}
\usepackage{float}

\newtheorem{Theorem}{Theorem}

\newtheorem{Proposition}{Proposition}

\newtheorem{Lemma}{Lemma}

\newtheorem{Example}{Example}

\newtheorem{Remark}{Remark}

\newtheorem{Cor}{Corollary}

\newtheorem{Assumption}{Assumption}

\newtheorem{Definition}{Definition}

\DeclareMathOperator{\rank}{rank}

\makeatletter
\hypersetup{colorlinks=true}
\AtBeginDocument{\@ifpackageloaded{hyperref}
  {\def\@linkcolor{blue}
  \def\@anchorcolor{red}
  \def\@citecolor{red}
  \def\@filecolor{red}
  \def\@urlcolor{black}
  \def\@menucolor{red}
  \def\@pagecolor{red}
\begingroup
  \@makeother\`%
  \@makeother\=%
  \edef\x{%
    \edef\noexpand\x{%
      \endgroup
      \noexpand\toks@{%
        \catcode 96=\noexpand\the\catcode`\noexpand\`\relax
        \catcode 61=\noexpand\the\catcode`\noexpand\=\relax
      }%
    }%
    \noexpand\x
  }%
\x
\@makeother\`
\@makeother\=
}{}}
\makeatother

\IEEEoverridecommandlockouts
\begin{document}

\title{Fixed-Time Stable Gradient Flows: Applications to Continuous-Time Optimization}
\author{Kunal Garg, \IEEEmembership{Student Member, IEEE} and Dimitra Panagou, \IEEEmembership{Senior Member, IEEE}
\thanks{The authors are with the Department of Aerospace Engineering, University of Michigan, Ann Arbor, MI, USA; \texttt{\{kgarg, dpanagou\}@umich.edu}.}
\thanks{The authors would like to acknowledge the support of the Air Force Office of Scientific
Research under award number FA9550-17-1-0284.}}

\maketitle

\begin{abstract}   
Continuous-time optimization is currently an active field of research in optimization theory; prior work in this area has yielded useful insights and elegant methods for proving stability and convergence properties of the continuous-time optimization algorithms. This paper proposes novel gradient-flow schemes that yield convergence to the optimal point of a convex optimization problem within a \textit{fixed} time from any given initial condition for unconstrained optimization, constrained optimization, and min-max problems. It is shown that the solution of the modified gradient flow dynamics exists and is unique under certain regularity conditions on the objective function, while fixed-time convergence to the optimal point is shown via Lyapunov-based analysis. The application of the modified gradient flow to unconstrained optimization problems is studied under the assumption of gradient-dominance, a relaxation of strong-convexity. Then, a modified Newton's method is presented that exhibits fixed-time convergence under some mild conditions on the objective function. Building upon this method, a novel technique for solving convex optimization problems with linear equality constraints that yields convergence to the optimal point in fixed time is developed. 
More specifically, constrained optimization problems formulated as min-max problems are considered, and a novel method for computing the optimal solution in fixed-time is proposed using the Lagrangian dual. 
Finally, the general min-max problem is considered, and a modified scheme to obtain the optimal solution of saddle-point dynamics in fixed time is developed.
Numerical illustrations that compare the performance of the proposed method against Newton's method, rescaled-gradient method, and Nesterov's accelerated method are included to corroborate the efficacy and applicability of the modified gradient flows in constrained and unconstrained optimization problems.   
\end{abstract}

\section{Introduction}
\subsection{Motivation}
Study of continuous-time optimization methods has been a very important part of the optimization theory from very early days \cite{brown1989some}. Research in this area continues to this day with the aim of developing and studying differential equations that model the commonly used discrete-time optimization algorithms \cite{wibisono2015accelerated,wibisono2016variational,su2014differential}. Establishing connections between ordinary differential equations (ODEs) and optimization has been an active topic of interest, see \cite{su2014differential,cherukuri2017saddle,dhingra2018proximal} and the references therein. The theory of ODEs offers useful insights into optimization theory and the corresponding techniques \cite{su2014differential}; some of the notable examples, as listed in \cite{su2014differential}, include linear regression via ODEs induced by linearized Bregman iteration algorithm \cite{osher2016sparse} and a continuous-time Nesterov-like accelerated algorithm in the context of control design \cite{durr2012class}. The continuous-time perspective of optimization problems provides simple and elegant proofs for the convergence of solutions to the equilibrium points using Lyapunov stability theory \cite{krichene2015accelerated}; this has led to further studies in unconstrained optimization \cite{brown1989some,helmke2012optimization}, constrained optimization \cite{schropp2000dynamical}, and more recently, saddle-point dynamics \cite{cherukuri2017saddle,cherukuri2017role}. It is worth noticing that while there is a lot of work on continuous-time optimization, most of it addresses asymptotic or exponential convergence of the solutions to the optimal point, i.e., convergence as time tends to infinity. In this paper, novel continuous-time optimization schemes are developed that possess fixed-time convergence guarantees, i.e., guarantees that the solutions of the considered ODEs converge to the optimal point of the corresponding optimization problem within a fixed time that is independent of the initial conditions.

\subsection{Gradient flows: theory and applications}
It is well known that the strict minima of a locally convex function $f:\mathbb R^n\rightarrow \mathbb R$ are stable equilibria of the \textit{gradient-flow} (GF) dynamics $\dot x = -\nabla f(x)$, and that, if the sub-level sets of $f$ are compact, then the trajectories converge asymptotically to the set of critical points of $f$. In recent years, GFs have been employed in a wide range of applications, including image processing \cite{clarenz2001relations} and motion planning \cite{cortes2006finite}. 
Details on design of GFs for optimization problems can be found in  \cite{helmke2012optimization}; for an overview of convex optimization, the reader is referred to \cite[Chapter 4-5]{boyd2004convex}.

There is a plethora of work on asymptotic convergence analysis of GF, for an overview, see \cite{wibisono2015accelerated,wibisono2016variational}. Recent work, for example \cite{su2014differential}, has focused on exponential stability of the GF based methods. The \textit{strong} or \textit{strict convexity} of the objective function is a standard assumption for exponential stability. As shown in \cite{karimi2016linear}, the condition can be relaxed by assuming that the objective function satisfies the  \textit{Polyak-{\L}ojasiewicz inequality} (PL inequality), i.e., the objective function is \textit{gradient dominated}. In \cite{nesterov2008accelerating}, the authors develop cubic regularization of Newton's method with super-linear convergence rate. Other accelerated methods include \textit{Bregman-Lagrangian} flows \cite{wibisono2016variational}, where instead of standard gradient flow, Euler-Lagrange equations for the Bregman-Lagrange flow are studied for super-linear convergence.

Another set of problems where GF is used is the \textit{saddle-point dynamics} for \textit{min-max problems}, where a multivariate function needs to be minimized over one set of variables and maximized over another set of variables. Saddle-point dynamics and its variations have been used extensively in the design and analysis of distributed feedback controllers \cite{wang2011control} and optimization algorithms in several domains, including active power loss minimization \cite{ma2013distributed}, network optimization \cite{feijer2010stability}, and zero-sum games \cite{gharesifard2013distributed} (see \cite{benzi2005numerical} for a detailed presentation on various applications where saddle-point dynamics naturally arise). 
In \cite{feijer2010stability}, the authors show asymptotic stability of the saddle-point dynamics, and apply the developed methods to internet-congestion control and network-utility maximization. 
More recently, in \cite{duchi2018minimax}, the authors develop general techniques for deriving minimax bounds under local differential privacy constraints. Lagrangian based primal-dual problems are another set of problems where min-max problems naturally arise. In \cite{cherukuri2017role,cherukuri2017saddle}, the authors discuss the conditions under which the saddle-point dynamics exhibit global asymptotic convergence. In \cite{qu2019exponential,dhingra2018proximal}, the authors show global exponential stability of the gradient-based method for primal-dual gradient dynamics under strong convexity-concavity assumption. 

\subsection{Finite-time and fixed-time stability}
In the seminal work \cite{bhat2000finite}, the authors introduced the notion of finite-time stability (FTS), where the convergence of the solutions to the equilibrium is guaranteed within a finite time, in contrast to asymptotic or exponential convergence where the solutions converge as time goes infinity. 
The authors give sufficient conditions in terms of existence of a Lyapunov function for FTS. 
Under this notion, the settling time, or time of convergence, depends upon the initial condition. A stronger notion, called fixed-time stability (FxTS), is developed in \cite{polyakov2012nonlinear}, where the settling time is uniformly bounded for all initial conditions. {The authors in \cite{bhat2000finite,polyakov2012nonlinear} discuss also the robustness of a FTS and a FxTS equilibrium, respectively; they show that the convergence properties are preserved under a class of additive vanishing disturbances, and that the trajectories of a system with a FTS or FxTS equilibrium, under non-vanishing disturbances, converge to a smaller neighborhood of the equilibrium point, as compared to a system whose equilibrium is asymptotically or exponentially stable.}
The primary motivation of the work in this paper is to study FxTS of the gradient-based methods with applications to convex optimization problems.

\subsection{Related prior work on continuous-time optimization}
In \cite{cortes2006finite}, the authors introduce normalized gradient flows to show finite-time convergence of the solutions to the optimal point. The authors in \cite{chen2018convex} consider convex optimization problems with equality constraints under strong convexity of the objective function, and design a discontinuous dynamics that converges to the optimal solution in finite time. Finite-time distributed optimization is studied in \cite{song2016finite,pan2018distributed}, where the authors assume very specific initial conditions, such that the sum of the gradient of the objective functions is zero. In \cite{li2017fixed}, the authors design a sliding-mode based technique for distributed optimization with fixed-time convergence guarantees assuming that the objective functions are strongly convex. In \cite{sanchez2014fixed}, a method of finding the optimal solution of a linear program (LP) in fixed time is proposed. In \cite{feng2017finite} and \cite{santilli2018finite}, the authors design finite-time converging schemes for distributed optimization where the objective function is of sum of quadratic functions and strictly convex functions, respectively. 

In the aforementioned work \cite{cortes2006finite,li2017fixed,sanchez2014fixed,feng2017finite,santilli2018finite}, the resulting dynamics are discontinuous, and the solutions are understood in the sense of Filippov. While \cite{cortes2006finite} mentions that the paths traced by the discontinuous dynamics and the nominal gradient flow $\dot x = -\nabla f(x)$ are identical, none of the other papers show uniqueness of the solutions in forward-time for the considered discontinuous dynamics. In this work, modified gradient flows are designed with continuous dynamics, and existence and uniqueness of solutions for all times and for all initial conditions is proven. Moreover, 
In this paper, fixed-time convergence is considered, in contrast to the finite-time convergence in \cite{cortes2006finite,song2016finite,pan2018distributed,feng2017finite}, where the time of convergence grows unbounded with initial conditions. 

\subsection{Contributions of the paper}

In this paper, modified GF schemes are designed for unconstrained and constrained convex optimization problems, as well as for min-max problems, with \textit{fixed-time} convergence guarantees. 
In \cite{song2016finite,pan2018distributed}, the authors assume that the Hessian of the objective function is Lipschitz continuous. This assumption as well as the strong-convexity assumption in \cite{chen2018convex,li2017fixed} are relaxed in this work, and it is shown that fixed-time convergence can be guaranteed for a larger class of problems where the objective function satisfies the PL inequality. In contrast to \cite{cortes2006finite,sanchez2014fixed}, convex optimization problems with linear equality constraints are studied in this paper, and a novel method is proposed to obtain the optimal point in fixed time under certain conditions on the smoothness and convexity of the objective function. In summary, contributions of the paper are as follows:
\begin{itemize}
    \item[a.] \textbf{FxTS-GF for unconstrained optimization}: a novel GF scheme with fixed-time convergence guarantees is proposed for unconstrained convex optimization problems under gradient-dominance; 
    \item[b.] \textbf{FxTS Newton's method for unconstrained optimization}: a Newton's-based method with fixed-time convergence guarantees is developed for the minimization of a strictly-convex function; 
    \item[c.] \textbf{FxTS-GF for constrained optimization}: a novel method to solve constrained optimization problems with equality constraints in fixed time is proposed, when the conjugate of the objective function is known in closed-form;
    \item[d.] \textbf{FxTS saddle-point dynamics for min-max problems}: a modified saddle-point dynamics with fixed-time convergence guarantees is designed for min-max problems under some mild conditions; 
    \item[e.] \textbf{FxTS saddle-point dynamics for constrained optimization}: it is shown that the proposed modified saddle-point dynamics can be used to solve constrained optimization problems when the conjugate function is not known in closed-form.
    \item[f.] \textbf{Numerical illustrations}: various numerical examples are presented to demonstrate that the proposed method achieves super-linear, fixed-time convergence with Euler discretization. It is also demonstrated that in comparison to the nominal Newton's method, 
    and the rescaled gradient method in \cite{wibisono2016variational}, the proposed method achieves faster convergence in terms of number of iterations, while requiring lower computational (wall-clock) time, which corroborates the practical applicability of the proposed methods.
\end{itemize}
To the best of authors' knowledge, this is the first work that establishes FxTS of GF-based techniques, and demonstrates their application to nonlinear constrained optimization and saddle-point dynamics. 
Though the theory presented in this paper treats continuous-time dynamics, the discrete-time implementation manifests the applicability of the proposed method in practice. Recent work on rate-preserving schemes \cite{wibisono2016variational} and \textit{consistent-discretization} schemes \cite{polyakov2018consistent,polyakov2019consistent} for finite or fixed-time stable dynamical systems motivates the future work of studying discretization schemes for the proposed method, which would guarantee convergence of the solutions in finite or fixed number of steps for any initial condition (see Section \ref{Discussions} for a detailed discussion on the matter).

\subsection{Organization}
The paper is organized as follows: Section \ref{Sec OV} presents an overview of the theory of FTS and FxTS, as well as an overview of convex optimization. Section \ref{FT opt} presents modified gradient flows for unconstrained optimization and constrained optimization problems with linear constraints, and fixed-time convergence to the optimal point is shown. In Section \ref{FT SP}, the min-max problem is studied for the general saddle-point dynamics. In Section \ref{simulations}, three numerical examples are presented to corroborate applicability of the proposed method; namely, an instance of a support-vector machine, a quadratic program, and a min-max problem. Section \ref{Discussions} discusses the limitation of the theory of FTS or FxTS for continuous-time dynamical systems when it comes to discrete-time or discretized settings, and lays out the foundation for future work in the open areas. The conclusions and plans on future work are summarized in Section \ref{Conclusions}.

\section{Background and Preliminaries}\label{Sec OV}
\subsection{Notation}
The set of reals is denoted by $\mathbb R$. The Euclidean norm of $x\in \mathbb R^n$ is denote by $\|x\|$, and its transpose, by $x^T$. For a given function $f$, $f^\star$ denotes the optimal value of the objective function for the given optimization problem and $x^\star$ denotes the optimal point, i.e., $f(x^\star) = f^\star$. The conjugate of the function $f$ is denoted as $f^*$ and is defined as $f^*(y) =  \sup\limits_{x\in\mathbb R^n}(y^Tx-f(x))$\footnote{Note the difference between $^\star$ for optimality and $^*$ for the conjugate.}. The notation $f\in C^k(U,V)$ is used for a function $f:U\rightarrow V$, $U\subseteq \mathbb R^n, V\subseteq \mathbb R^m$ which is $k-$times continuously differentiable, and $f\in C^{1,1}_{loc}(U,V)$ is used to denote a continuously differentiable function whose gradient is locally Lipschitz continuous on $U$. The notation $\nabla f: \mathbb R^n\rightarrow \mathbb R^n$ and $\nabla^2f:\mathbb R^n\rightarrow \mathbb R^{n\times n}$ is used to denote the gradient and the Hessian of the function $f\in C^2(\mathbb R^n,\mathbb R)$, respectively. For a multivariate function $f\in C^2(\mathbb R^n,\mathbb R)$, the partial derivatives are denoted as $\nabla_{x_1} f(x) \triangleq \frac{\partial f}{\partial x_1}(x)$ and $\nabla_{x_1x_2} f(x) \triangleq \frac{\partial^2 f}{\partial x_1 \partial x_2}(x)$, where $x_1\in \mathbb R^{r_1}, x_2\in \mathbb R^{r_2}$, $r_1, r_2\leq n$. A positive definite (respectively, semi-definite) matrix $A\in \mathbb R^{n\times n}$ is denoted as $A\succ 0$ (respectively, $A\succeq 0$); for $A, B\in \mathbb R^{n\times n}$, $A\succ B$ (respectively, $A\succeq B$) is used when $A-B \succ 0$ (respectively, $A-B\succeq 0$). For sake of brevity, the argument $x$ of a function or its derivatives is dropped, whenever clear from the context.

\subsection{Mathematical preliminaries}
First, an overview of the mathematical preliminaries and some useful results are presented. Consider the system
\begin{equation}\label{ex sys}
\dot x = f(x),
\end{equation}
where $x\in \mathbb R^n$, $f: \mathbb R^n \rightarrow \mathbb R^n$ and $f(0)=0$. Assume that the solution to \eqref{ex sys} exists, is unique, and continuous for any $x(0)\in \mathbb R^n$, for all $t\geq 0$.

\begin{Definition}[\hspace{-0.1pt}\cite{bhat2000finite}]
The origin is said to be an FTS equilibrium of \eqref{ex sys} if it is Lyapunov stable and \textit{finite-time convergent}, i.e., for all $x(0) \in \mathcal N \setminus\{0\}$, where $\mathcal N$ is some open neighborhood of the origin, $\lim_{t\to T} x(t)=0$, where $T = T(x(0))<\infty$. The origin is said to be a globally FTS equilibrium if $\mathcal N = \mathbb R^n$.
\end{Definition}

\noindent Here, $T$ is called as the settling time.  Lyapunov conditions for FTS of the origin for system \eqref{ex sys} are as follows.

\begin{Lemma}[\hspace{-0.1pt}\cite{bhat2000finite}]\label{FTS Bhat}
Suppose there exists a positive definite function $V\in C^1(\mathcal D,\mathbb R)$, where $\mathcal D\subset\mathbb R^n$ is a neighborhood of the origin, constants $c>0$ and $\alpha \in (0, 1)$, and an open neighborhood $\mathcal{V}\subseteq \mathcal{D}$ of the origin such that 
 \begin{equation} \label{FTS Lyap}
     \dot V(x) + cV(x)^\alpha \leq 0, \; \forall x\in \mathcal{V}\setminus\{0\}.
 \end{equation}
Then, the origin is an FTS equilibrium of \eqref{ex sys}. Moreover, the settling time $T$ satisfies $T(x(0)) \leq \frac{V(x(0))^{1-\alpha}}{c(1-\alpha)}$.
\end{Lemma}

\noindent Under the notion of FTS, the settling time $T$ depends upon the initial condition $x(0)$ and grows unbounded as $\|x(0)\|$ increases. The notion of FxTS allows the settling time to remain upper bounded, independent of the initial condition.

\begin{Definition}[\hspace{-0.1pt}\cite{polyakov2012nonlinear}]
The origin is said to be an FxTS equilibrium of \eqref{ex sys} if it is globally FTS and the settling time $T(x(0))$ is uniformly bounded, i.e., there exists $\bar T<\infty$ such that $\sup_{x(0)\in \mathbb R^n}T(x(0))\leq \bar T$. 
\end{Definition}

\begin{Lemma}[\hspace{-0.1pt}\cite{polyakov2012nonlinear}]\label{FxTS TH}
Suppose there exist a positive definite function $V\in C^1(\mathcal D,\mathbb R)$, where $\mathcal D\subset\mathbb R^n$ is a neighborhood of the origin, for system \eqref{ex sys} such that 
\begin{equation}\label{dot V eq}
    \dot V(x) \leq -pV(x)^{\alpha}-qV(x)^{\beta}, \; \forall x\in \mathcal D\setminus\{0\},
\end{equation}
with $p,q>0$, $0<\alpha<1$ and $\beta>1$. Then, the origin of \eqref{ex sys} is FxTS with settling time 
\begin{equation}
    T \leq \frac{1}{p(1-\alpha)} + \frac{1}{q(\beta-1)}. 
\end{equation}
\end{Lemma}

Next, various notions of convexity, and first and second order conditions for convexity are summarized in the following lemma (see \cite[Chapter 3]{boyd2004convex} for more details).

\begin{Definition}
A function $f\in C^1(D,\mathbb R)$, where $D\subset\mathbb R^n$ is a convex set, is called
\begin{itemize}
    \item[-] \textbf{Convex} if for all $x,y\in D$ and all $\alpha\in (0,1)$, $ f(\alpha x+(1-\alpha)y)\leq \alpha f(x)+(1-\alpha)f(y)$;
     \item[-] \textbf{Concave} if $(-f)$ is convex;
     \item[-] \textbf{Strictly convex}  if for all $x,y\in D$ and all $\alpha\in (0,1)$, $
     f(\alpha x+(1-\alpha)y)< \alpha f(x)+(1-\alpha)f(y)$;
    \item[-] \textbf{$m$-Strongly convex} if there exists $m>0$ such that $f(y)\geq f(x) +\nabla f(x)^T(y-x)+\frac{m}{2}\|x-y\|^2$, for all $x,y\in D$;
    \item[-] \textbf{$\beta$-Strongly smooth} if for all $x,y\in D$, $f(y)\leq f(x) +\nabla f(x)^T(y-x)+\frac{\beta}{2}\|x-y\|^2$, where $\beta>0$.
\end{itemize}
\end{Definition}

\begin{Lemma}\label{Lemma convexity}
\textbf{First-order conditions}: A function $f\in C^1(D,\mathbb R)$, where $D\subset\mathbb R^n$ is a convex set, is
\begin{itemize}
    \item[-] \textbf{Convex} if and only if for all $x,y\in D$, $f(y) \geq f(x) + \nabla f(x)^T(y-x)$;
    \item[-] \textbf{$m$-Strongly convex} if and only if for all $x,y\in D$, $(\nabla f(x)-\nabla f(y))^T(x-y)\geq m\|x-y\|^2$ for some $m>0$.
\end{itemize}
\textbf{Second-order conditions}: A function $f\in C^2(D,\mathbb R)$, $D\subset\mathbb R^n$ is
\begin{itemize}
    \item[-] \textbf{Convex} if and only if for all $x\in D$, $\nabla^2 f(x) \succeq 0$;
     \item[-] \textbf{Strictly convex}  if for all $x\in D$, $\nabla^2 f(x) \succ 0$;
    \item[-] \textbf{$m$-Strongly convex} if and only if for all $x\in D$, $\nabla^2 f(x) \succeq mI$ for some $m>0$.
    \item[-] \textbf{$\beta$-Strongly smooth} if and only if for all $x\in D$, $\nabla^2 f(x) \preceq\beta I$ for some $\beta>0$.
\end{itemize}
\end{Lemma}

\noindent It follows that strong-convexity implies strict-convexity, which implies convexity. 

\begin{Definition}
A function $F:D_1\times D_2\rightarrow \mathbb R$, where $D_1\subset\mathbb R^n, D_2\subset\mathbb R^m$, is called locally convex-concave (respectively, locally strongly or locally strictly convex-concave) if for any fixed $\bar z\in U_z\subset D_2$, $F(x,\bar z)$ is convex (respectively, strongly or strictly convex) for all $x\in U_x\subset D_1$, and for any fixed $\bar x\in U_x\subset  D_1$, $F(\bar x, z)$ is concave (respectively, strongly or strictly concave) for all $z\in U_z\subset D_2$. 
\end{Definition}

\section{FxTS in Optimization}\label{FT opt}
In this section, novel gradient flow schemes are proposed for unconstrained and constrained convex optimization problems. First, unconstrained optimization problems are considered.

\subsection{Unconstrained optimization: FTS scheme}
Consider the unconstrained minimization problem
\begin{equation}\label{op 1}
    \min_{x\in \mathbb R^n} f(x),
\end{equation}
where $f:\mathbb R^n\rightarrow\mathbb R$. The following assumption is made about the problem \eqref{op 1}. 

\begin{Assumption}\label{assum exist xst}
The minimum value of $f(x)$ is attained, i.e., there exists $x^\star\in \mathbb R^n$ such that $-\infty<f^\star = f(x^\star)$. 
\end{Assumption}

\begin{Remark}
For \eqref{op 1}, Assumption \ref{assum exist xst} is a necessary condition for convergence of gradient-based methods to an optimal solution. Coercivity, or equivalently, compactness of the sub-level sets of the convex function $f$ is a sufficient condition to guarantee existence of a minimizer \cite[Chapter 2]{beck2014introduction}.
\end{Remark}

\begin{Lemma}[\hspace{-0.1pt}\cite{boyd2004convex}] \label{optim fixed point}
If $f$ is convex and differentiable, then a point $x^\star$ is the global optimal point of the function $f$ if and only if $\nabla f(x^\star) = 0$. Furthermore, if $f$ is strictly convex, then the optimal point $x^\star$ is unique. 
\end{Lemma}

There has been a lot of research on developing discrete-time optimization schemes with convergence rate faster than linear (see \cite{wibisono2016variational,nesterov2008accelerating} and references therein). The continuous variant of such discrete-time schemes are also studied by various authors. In \cite{wibisono2016variational}, the authors discuss the following scheme
\begin{equation}\label{p flow}
    \dot x = -\frac{\nabla f(x)}{\|\nabla f(x)\|^{\frac{p-2}{p-1}}},
\end{equation}
where $p>2$ as a modification of GF. It is shown that the convergence rate for the solutions of \eqref{p flow} is given as
\begin{equation}\label{p conv}
    f(x(t))-f^{\star} \leq O\left(\frac{1}{t^{p-1}}\right),
\end{equation}
under the assumption that the level-sets of $f(x)$ are bounded. The flow in \eqref{p flow} is referred to as \textit{rescaled} GF in the subsequent text. In this subsection, it is shown that the optimal point of \eqref{op 1} is actually an FTS equilibrium of \eqref{p flow}. Then, in the subsequent subsections, modifications of \eqref{p flow} are presented with fixed-time convergence guarantees. 

\begin{Theorem}\label{p flow FT}
If $f\in C^2(\mathbb R^n, \mathbb R)$ is $k$-strongly convex for some $k>0$, then the trajectories of \eqref{p flow} converge to the optimal point $x^\star$ in finite time $T  = T(x(0))$, for any $p>2$.
\end{Theorem}
\begin{proof}
First, using Lemma \ref{Lemma convexity}, one has that strong convexity of $f$ implies that the optimal solution $x^\star$ of \eqref{op 1} is unique and satisfies $\nabla f(x^\star) = 0$. Choose $V(x) = \frac{1}{2}\|\nabla f(x)\|^2$ as the candidate Lyapunov function. $k$-strong convexity of $f$ implies that $\nabla^2f(x)\succeq kI$ for all $x\in \mathbb R^n$. Using this, one obtains that the time derivative of $V$ along \eqref{p flow} satisfies
\begin{align*}
    \dot V & = \nabla f^T\nabla^2f\dot x = -\nabla f^T\nabla^2f\frac{\nabla f}{\|\nabla f\|^\frac{p-2}{p-1}}\\
    & \leq -k \nabla f^T\frac{\nabla f}{\|\nabla f\|^\frac{p-2}{p-1}} = -k\|\nabla f\|^{2-\frac{p-2}{p-1}}  \\
    & = -k\|\nabla f\|^{\frac{p}{p-1}}  =  -k2^{\frac{p}{2(p-1)}}V^{\frac{p}{2(p-1)}}.
\end{align*}
Define $k_1 = k2^{\frac{p}{2(p-1)}}>0$ and $\beta_1 = \frac{p}{2(p-1)}$, so that $0<\beta_1<1$, and one obtains $ \dot V \leq -k_1V^{\beta_1}$. Using Lemma \ref{FTS Bhat}, one obtains that $\|\nabla f(x(t))\| = 0$ for all $t\geq T$ where $T\leq \frac{V(x(0))^{1-\beta_1}}{k_1(1-\beta_1)}$. Since the function is strongly convex, the sublevel sets of the norm of $\nabla f$ are bounded. Thus, $V$ is radially unbounded, and hence, the result holds for all $x(0)\in \mathbb R^n$. 
\end{proof} 

\begin{Remark}
For a given $p>2$, denote the bound on the time of convergence as $\bar T_p$. As noted in \cite{wibisono2016variational}, the limiting case of \eqref{p flow} as $p\rightarrow \infty$, called normalized GF, is studied in \cite{cortes2006finite}, and it is shown that the time of convergence is upper bounded by $\frac{1}{k}\|\nabla f(x(0))\|$ under the assumption of strong-convexity. The same bound on the time of convergence is recovered by allowing $p\rightarrow \infty$ in the bound of the settling time $\bar T_{\infty}$ in Theorem \ref{p flow FT}. 
Note that for initial values farther away from the optimal point, the upper bound on the time of convergence satisfy $\bar T_p\leq \bar T_{\infty}$, i.e., the upper bound for $2<p<\infty$ is lower than $\frac{1}{k}\|\nabla f(x(0)\|$.
\end{Remark}

It is clear from the expression of the bound on the settling time $T$ in Theorem \ref{p flow FT} that it grows unbounded as the distance of $x(0)$ increases from the optimal point $x^\star$. Inspired from \eqref{p flow} and noting its finite-time convergence guarantees, a modified GF is designed in this subsection with fixed-time convergence guarantees, so that the the optimal point of \eqref{op 1} can be obtained within a fixed time for any $x(0)\in \mathbb R^n$. 

\subsection{Unconstrained optimization: FxTS-GF scheme}
Consider the flow equation
\begin{equation}\label{fixed flow}
    \dot x = -c_1\frac{\nabla f(x)}{\|\nabla f(x)\|^\frac{p_1-2}{p_1-1}}-c_2\frac{\nabla f(x)}{\|\nabla f(x)\|^\frac{p_2-2}{p_2-1}},
\end{equation}
where $c_1,c_2>0$, $p_1>2$ and $1<p_2<2$. In what follows, \eqref{fixed flow} is referred to as \textit{FxTS-GF}. First, it is shown that the equilibrium points of the right-hand side of \eqref{fixed flow} are critical points\footnote{Recall that a point $x$ is called a critical point of a $C^1$ function $f$ if $\nabla f(x) =0$.} of the function $f$, and that the dynamics in \eqref{fixed flow} is continuous for all $x\in \mathbb R^n$.\footnote{Note that the dynamics in \eqref{fixed flow} is defined as $\dot x = 0$ when $\nabla f(x) = 0$.}

\begin{Lemma}\label{eq point fixed point}
A point $\bar x\in \mathbb R^n$ is an equilibrium point of \eqref{fixed flow} if and only if $\|\nabla f(\bar x)\| = 0$.  
\end{Lemma}

\begin{Lemma}\label{cont fixed flow}
If $f\in C^{1,1}_{loc}(\mathbb R^n, \mathbb R)$, then the right-hand side of \eqref{fixed flow} is continuous for all $x\in \mathbb R^n$. 
\end{Lemma}

\noindent Proofs of Lemma \ref{eq point fixed point} and Lemma \ref{cont fixed flow} are given in Appendix \ref{app proof Lemma eq fixed point}, \ref{app proof Lemma cont}, respectively. Next, it is shown that the solutions of \eqref{fixed flow} exist and are unique in forward time.

\begin{Proposition}\label{prop exist unique}
If the function $f\in C^{1,1}_{loc}(\mathbb R^n, \mathbb R)$ is convex with unique minimizer, then for any $x(0)\in \mathbb R^n$, the solution of \eqref{fixed flow} exists and is unique for all $t\geq 0$. 
\end{Proposition}

\begin{proof}
The result \cite[Proposition 2]{garg2019fixedMP} shows existence and uniqueness of solutions for an ODE of the form \eqref{fixed flow} with $X = \nabla f$. Per the statement of the Theorem, $f\in C^{1,1}_{loc}$, which implies that $\nabla f$ is locally Lipschitz. Also, with $0<\alpha_1 = 1-\frac{p_1-2}{p_1-1}<1$ and $\alpha_2 =1- \frac{p_2-2}{p_2-1}>1$, all the conditions of \cite[Proposition 2]{garg2019fixedMP} are satisfied. Hence, the solution of \eqref{fixed flow} exists in forward time and is unique, for any initial condition $x(0)\in \mathbb R^n$. \end{proof}

\noindent Before presenting the main result of this subsection, the following required assumptions on the objective function $f$ is made. 

\begin{Assumption}{(\textbf{Gradient dominated})} \label{f assum 2}
The function $f\in C^{1,1}_{loc}(\mathbb R^n,\mathbb R)$ has a unique minimizer $x = x^\star$ and satisfies the Polyak-\L ojasiewicz (PL) inequality, or is gradient dominated, with $\mu_f>0$, i.e., for all $x \in \mathbb R^n$, 
\begin{equation}\label{PL ineq}
    \frac{1}{2}\|\nabla f(x)\|^2\geq \mu_f(f(x)-f^{\star}).
\end{equation}
\end{Assumption}

\begin{Remark}\label{PL remark}
Strong convexity of the objective function is a  standard assumption used in literature to show exponential convergence for gradient flows. As noted in \cite{karimi2016linear}, PL inequality is the weakest condition among other similar conditions popularly used in the literature to show linear convergence in discrete-time (exponential, in continuous-time) gradient-based algorithms. Particularly, a strongly convex function $f\in C^1(\mathbb R^n, \mathbb R)$ satisfies PL inequality. {Note that under this assumption, it is not required for the objective function $f$ to be convex.}
\end{Remark}

\noindent It is shown in \cite[Theorem 2]{karimi2016linear} that satisfaction of PL inequality implies that the function $f$ has \textit{quadratic growth}, i.e., 
\begin{align}\label{eq: QG ineq}
    f(x)-f^\star\geq \frac{\mu_f}{2}\|x-x^\star\|^2.
\end{align}
for all $x$, where $\mu_f$ is as defined in \eqref{PL ineq}. The following result can now be stated. 

\begin{Theorem}\label{Fixed time GF}
If the objective function $f$ satisfies 
Assumptions \ref{assum exist xst} and \ref{f assum 2}, then, the trajectories of \eqref{fixed flow} converge to the optimal point $x^\star$ in fixed time for all $x(0)$.
\end{Theorem}
\begin{proof}
First, note that the solutions of \eqref{fixed flow} exist and are unique for all $x(0) \in \mathbb R^n\setminus \{x^\star\}$ \cite[Proposition 2]{garg2019fixedMP}.
{Note that the convexity assumption in \cite[Proposition 2]{garg2019fixedMP}, and consequently, in Proposition \ref{prop exist unique}, is sufficient to show uniqueness of the solutions starting from the minimizer, i.e., for the case when $x(0) = x^\star$. For the case when the function $f$ is not convex, the uniqueness of the solution can still be guaranteed if the equilibrium point is attractive. Assume that for \eqref{fixed flow} with a non-convex $f$, there exist two solutions, namely the trivial solution $x_1(\cdot) = x^\star$, and another solution such that $x_2(\cdot)\neq x^\star$, with $x_1(0) = x_2(0) = x^\star$. Define $V_1 = f(x_1(t))-f(x_2(t)) = f(x_1(t))-f^\star$ as a candidate Lyapunov function. Note that since $x_1(t) \neq x^\star$, $f(x_1(t))>f^\star$, and hence, $V_1$ is a positive definite function. The time derivative of $V_1$ reads $\dot V_1 = \nabla f(x_1(t))^T\dot x_1(t)$, which along the trajectories of \eqref{fixed flow} satisfies $\dot V_1 \leq 0$. Since $V_1(0)= f(x_1(0))-f(x_2(0)) = 0$, and $\dot V_1\leq 0$, using \cite[Theorem 3.15.1]{agarwal1993uniqueness}, one obtains that $x_1(t) = x_2(t) = x^\star$, i.e., the solution of \eqref{fixed flow} is unique for $x(0)= x^\star$.}

Now, consider the candidate Lyapunov function $V(x) = \frac{1}{2}(f(x)-f^{\star})^2$. From \eqref{eq: QG ineq}, it is clear that $V$ is radially unbounded. The time derivative along the trajectories of \eqref{fixed flow} reads
\begin{equation*}\begin{split}
    \dot V & = (f-f^{\star})(\nabla f)^T\left( -c_1\frac{\nabla f}{\|\nabla f\|^\frac{p_1-2}{p_1-1}}-c_2\frac{\nabla f}{\|\nabla f\|^\frac{p_2-2}{p_2-1}}\right)\\
    & = -c_1 (f-f^{\star})\|\nabla f\|^{2-\frac{p_1-2}{p_1-1}}-c_2 (f-f^{\star})\|\nabla f\|^{2-\frac{p_2-2}{p_2-1}}\\
    & = -c_1(f-f^{\star})\|\nabla f\|^{\alpha_1}-c_2(f-f^{\star})\|\nabla f\|^{\alpha_2}\\
    & \overset{\eqref{PL ineq}}{\leq} -c_1(2\mu_f)^{\frac{\alpha_1}{2}}(f-f^{\star})^{1+\frac{\alpha_1}{2}} -c_2(2\mu_f)^{\frac{\alpha_2}{2}}(f-f^{\star})^{1+\frac{\alpha_2}{2}}\\
    & = -c_12^{\frac{2+3\alpha_1}{4}}\mu_f^{\frac{\alpha_1}{2}}V^{\frac{2+\alpha_1}{4}}-c_22^{\frac{2+3\alpha_2}{4}}\mu_f^{\frac{\alpha_2}{2}}V^{\frac{2+\alpha_2}{4}} \\
    & =  -k_1V^{\frac{2+\alpha_1}{4}}-k_2V^{\frac{2+\alpha_2}{4}},
\end{split}\end{equation*}
where $\alpha_1 = 2-\frac{p_1-2}{p_1-1},\alpha_2 = 2-\frac{p_2-2}{p_2-1}, k_1 =c_12^{\frac{2+3\alpha_1}{4}}\mu_f^{\frac{\alpha_1}{2}}$ and $k_2 = c_22^{\frac{2+3\alpha_2}{4}}\mu_f^{\frac{\alpha_2}{2}}$. Since $\alpha_1<2$ and $\alpha_2>2$, one has $\frac{2+\alpha_1}{4}<1$ and $\frac{2+\alpha_2}{4}>1$. Hence, from Lemma \ref{FxTS TH}, one obtains that for $t\geq T_1$, $f(x(t)) = f^\star$, which is equivalent to $x(t) = x^\star$ under Assumption \ref{f assum 2}, where $T_1 \leq \frac{4}{k_1(2-\alpha_1)}+\frac{4}{k_2(\alpha_2-2)}$. Hence, the trajectories of \eqref{fixed flow} converge to the optimal point $x^\star$ of \eqref{op 1} in fixed time. 
\end{proof}

\begin{Remark}
Note that the difference between the proposed modified gradient flow \eqref{fixed flow} and the rescaled gradient flow \eqref{p flow} is the second term with exponent $1<p_2<2$. This term results into the second term $-V^\beta$ in \eqref{dot V eq}, while the first term, with exponent $p_1>2$, results into the first term $-V^\alpha$ in \eqref{dot V eq}. Intuitively, compared to the exponential convergence condition $\dot V\leq -V$, terms $V^\beta$ and $V^\alpha$ dominate the term $V$ when $V$ is large and small, respectively, resulting into accelerated convergence for both small and large initial distance from the equilibrium point. Since \eqref{p flow} contains only the first term, which dominates when $V$ is small, the time of convergence, though finite, grows larger as the initial distance from the equilibrium increases.
\end{Remark}

It is showed that the FxTS-GF in \eqref{fixed flow} can be used to find the optimal solution of \eqref{op 1} in fixed time. As mentioned in Remark \ref{PL remark}, Assumption \ref{f assum 2} is a relaxation used to show exponential convergence of gradient flow. Hence, all such problems which have been shown to have exponential convergence under strong-convexity can be solved within fixed time using \eqref{fixed flow}. It is easy to show that if a function $f:\mathbb R^m\rightarrow \mathbb R$ is strongly convex, then the function $g:\mathbb R^n\rightarrow \mathbb R$, defined as $g(x) = f(Ax)$, $A\in \mathbb R^{n\times m}$, is strongly convex if $A$ is full row-rank. If matrix $A$ is not full row-rank, then $g$ may not be strongly convex. On the other hand, as shown in \cite[Appendix 2.3]{karimi2016linear}, $g$ still satisfies PL inequality for any matrix $A$. Below, an example of an important class of problems is given for which, the objective function satisfies PL inequality (see \cite{karimi2016linear} for more examples on useful functions that satisfy PL inequality).

\begin{Example}\label{Example: PL inequality}
\textbf{Least squares}:  Consider the optimization problem 
\begin{equation}\label{PL example 1}
    \min_{x\in \mathbb R^n} f(Ax) = \|Ax-b\|^2,
\end{equation}
where $x\in \mathbb R^n, A\in \mathbb R^{n\times n}$ and $b\in \mathbb R^n$. Here, the function $f(x) = \|x-b\|^2$ is strongly-convex, and hence, $g(x) = \|Ax-b\|^2$ satisfies PL inequality for any matrix $A$.
\noindent \textbf{Linear regression:} Consider the optimization problem 
\begin{equation}\label{PL example 2}
    \min_{x\in \mathbb R^n} f(Ax) = \sum_{i = 1}^m\log(1+b_ia_i^Tx),
\end{equation}
where $x\in \mathbb R^n, a_i\in \mathbb R^{n}$ and $b\in \mathbb R$ for $i = 1,2, \dots, m$. Here, the function $g(x) = f(Ax)$ satisfies PL inequality for any matrix $A$.
\end{Example}

The objective functions in \eqref{PL example 1} and \eqref{PL example 2}
satisfy PL inequality, but need not be strongly convex for any matrix $A$; if additionally, uniqueness of the optimal solutions of \eqref{PL example 1}  and \eqref{PL example 2} 
is assumed, one can use \eqref{fixed flow} to find the optimal solutions for \eqref{PL example 1}  and \eqref{PL example 2}, respectively, 
in fixed time. These are important classes of functions in machine learning problems. Next, the modification of the Newton's method based GF is presented to guarantee FxTS for a class of functions that do not satisfy Assumption \ref{f assum 2}.

\subsection{FxTS Newton's method}

In this subsection, a modified Newton's method is proposed that guarantees fixed-time convergence to the optimal point. The \textit{nominal} Newton's method is defined as
\begin{equation}\label{NM orig}
    \dot x = -\left(\nabla^2 f(x)\right)^{-1} \nabla f(x).
\end{equation}
It is well-known that under certain conditions on the function $f$, \eqref{NM orig} can achieve exponential convergence. The following assumption is made about the objective function $f$. 

\begin{Assumption}\label{assum inv f}
The function $f\in C^2(\mathbb R^n,\mathbb R)$ is strictly convex.
\end{Assumption}

Per Assumption \ref{assum inv f}, $\nabla^2f \succ 0$, which implies that the Hessian is invertible\footnote{This is needed so that the right-hand side in the Newton's method is well-defined.}, and with Assumption \ref{assum exist xst}, using Lemma \ref{optim fixed point}, one has that the optimal point $x^\star$ is unique. Note that if $f$ satisfies Assumption \ref{assum inv f}, it is not necessary that it satisfies Assumption \ref{f assum 2}. So, for \eqref{op 1} with this class of functions, fixed-time convergence cannot be guaranteed using \eqref{fixed flow}. Hence, another modified GF is proposed so that fixed-time convergence for this class of functions can be guaranteed. Consider the flow equation for FxTS Newton's method
\begin{equation}\label{FT Newton}
    \dot x = -(\nabla^2 f(x))^{-1}\left(c_1\frac{\nabla f(x)}{\|\nabla f(x)\|^\frac{p_1-2}{p_1-1}}+c_2\frac{\nabla f(x)}{\|\nabla f(x)\|^\frac{p_2-2}{p_2-1}}\right),
\end{equation}
where $c_1, c_2>0$, $p_1>2$ and $1<p_2<2$. The following result can now be stated. 

\begin{Theorem}\label{NT FT}
If $f$ satisfies Assumptions \ref{assum exist xst} and \ref{assum inv f}, then the trajectories of \eqref{FT Newton} converge to the optimal point $x^\star$ in fixed time $T_{NM}$ 
for any initial condition $x(0)\in \mathbb R^n$. 
\end{Theorem}
\begin{proof}
Note that under Assumption \ref{assum inv f}, per Proposition \ref{prop exist unique}, solutions of \eqref{FT Newton} exist and are unique for all $x(0) \in \mathbb R^n$. Consider the Lyapunov function $V(x) = \frac{1}{2}\|\nabla f(x)\|^2$. Since the sub-level sets of norm of its gradient $\nabla f$ are bounded for a strictly convex function under Assumption \ref{assum exist xst} (\cite[Corollary 8.7.1]{rockafellar2015convex}), the candidate Lyapunov function $V$ is radially unbounded.  The time derivative of this function along the trajectories of \eqref{FT Newton} reads
\begin{equation*}\begin{split}
    \dot V & = (\nabla f)^T(\nabla^2 f) \dot x\\
    & = -(\nabla f)^T\left(c_1\frac{\nabla f}{\|\nabla f\|^\frac{p_1-2}{p_1-1}} + c_2\frac{\nabla f}{\|\nabla f\|^\frac{p_2-2}{p_2-1}}\right)\\
    & = -c_1\|\nabla f\|^{2-\frac{p_1-2}{p_1-1}}-c_2\|\nabla f\|^{2-\frac{p_2-2}{p_2-1}} \\
    & \leq  -c_12^{\frac{\alpha_1}{2}} V^{\frac{\alpha_1}{2}}-c_22^{\frac{\alpha_2}{2}} V^{\frac{\alpha_2}{2}},
\end{split}\end{equation*}
where $\alpha_1 = 2-\frac{p_1-2}{p_1-1}$ and $\alpha_2 = 2-\frac{p_2-2}{p_2-1}$. Since $p_1>2$ and $1<p_2<2$ one obtains that $1<\alpha_1<2$ and $\alpha_2>2$. Hence, using Lemma \ref{FxTS TH}, one obtains that the trajectories of \eqref{FT Newton} converge to the optimal point $x^\star$ in the fixed time $T_{NM}$ for all $x(0)\in \mathbb R^n$, where $T_{NM}\leq \frac{2^{1-\frac{\alpha_1}{2}}}{c_1(2-\alpha_1)} +  \frac{2^{1-\frac{\alpha_2}{2}}}{c_2(\alpha_2-2)}$. 
\end{proof}

While strongly-convex functions satisfy PL inequality \cite{karimi2016linear}, strictly convex functions do not satisfy PL inequality in general. So, for convex optimization problems with strictly convex objective functions, that do not satisfy Assumption \ref{f assum 2}, \eqref{FT Newton} can be used to find the optimal solution of \eqref{op 1} within a fixed time. One example is the class of \textit{quartic} functions, which can be used to reformulate standard QP with sign constraints as uncontrained optimization problem \cite{bomze2005quartic}.

\begin{Example}\label{Example: QP NT}
Consider the optimization problem 
\begin{equation}\label{QP example NT method}
\begin{split}
    \min_x & \; x^TQx + c^Tx, \\
    \textnormal{s.t.} & \; x_i\geq 0, \; i = 1, 2, \dots, n,
\end{split}
\end{equation}
where $x, c\in \mathbb R^n$ and $Q\in \mathbb R^{n\times n}$ is a positive definite matrix. Let $z \in \mathbb R^n$ be defined as $x_i = z_i^2$, to get rid of the sign contraints, and re-write \eqref{QP example NT method} in terms of $z$
\begin{equation}\label{quartic opt Z}
    \min_z \; z^TZQZz + c^TZz, \\
\end{equation}
where $Z\in \mathbb R^{n\times n}$ is a diagonal matrix consisting of elements $z_i$, i.e., 
\begin{equation*}
    Z_{ij} = \begin{cases}z_i, & i = j;\\
    0, & i\neq j,
    \end{cases}
\end{equation*}
for $i, j = 1, 2, \dots, n$. The optimal solution $\bar x$ of \eqref{QP example NT method} is given by $\bar x_i = \bar z^2_i$, where $\bar z$ is the optimal solution of \eqref{quartic opt Z}. 
 \end{Example}

\noindent It is clear that the objective function in \eqref{quartic opt Z} is a quartic, is not strongly convex and may not satisfy PL inequality. Nevertheless, it is strictly convex and hence, \eqref{FT Newton} can be used to find the optimal point of \eqref{quartic opt Z} within fixed time. 

Upto now, unconstrained minimization problems have been considered. Next, constrained minimization problems are studied, and FxTS-GF based methods are proposed with fixed-time convergence guarantees.

\subsection{Convex optimization with linear equality constraints}\label{Sec Eq Const}
Consider the optimization problem 
\begin{equation}\label{Equal cons}
\begin{split}
        \min_{x\in \mathbb R^n} &\; f(x), \\
        \textrm{s.t.} \; & Ax = b,
\end{split}
\end{equation}
where $f:\mathbb R^n\rightarrow \mathbb R$ is convex, $A\in \mathbb R^{m\times n}$ and $b\in \mathbb  R^m$. 

\begin{Assumption}\label{assum A full rank}
The matrix $A$ is full row-rank and the objective function $f$ is coercive.
\end{Assumption}

\begin{Remark}
Assumption \ref{assum A full rank} is commonly used in constrained optimization \cite{qu2019exponential}; the matrix $A$ being full row-rank guarantees that the  feasible set is non-empty and closed, and thus, coercivity of the convex function $f$ guarantees that the solution of \eqref{Equal cons} exists \cite[Chapter 2]{beck2014introduction}. 
\end{Remark}

\noindent Define $g: \mathbb R^m\rightarrow \mathbb R$ as
\begin{equation}\label{g func}
    g(\nu) = \inf_{x\in \mathbb R^n} (f(x) + \nu^T(Ax-b)),
\end{equation} 
so that the dual problem (see \cite[Chapter 5]{boyd2004convex}) for \eqref{Equal cons} is given by
\begin{equation}\label{max optim}
    \sup_{\nu\in \mathbb R^m} \; g(\nu).
\end{equation}
Using \eqref{g func}, rewrite \eqref{max optim} as
\begin{equation}\label{SP x nu}
     \sup_{\nu\in \mathbb R^m} \inf_{x\in \mathbb R^n} L(x,\nu) \triangleq f(x) + \nu^T(Ax-b).
\end{equation}
It is clear that \eqref{SP x nu} is a saddle-point problem, where the function $L(x,\nu)$ needs to be minimized over $x$ and maximized over $\nu$. Using \eqref{g func}, one obtains (\hspace{-0.1pt}\cite[Section 5.1]{boyd2004convex})
\begin{equation}\label{dual opt}
    g(\nu) = -\nu^Tb - f^*(-A^T\nu),
\end{equation}
where $f^*:\mathbb R^n \rightarrow\mathbb R$ is the conjugate of $f$.\footnote{Since the considered space is the finite-dimensional vector space $\mathbb R^n$ with the Euclidean norm, the dual space is $\mathbb R^n$ with the dual norm $\|\cdot\|_* = \|\cdot\|$.} Note that the function $f^*$ is always convex, whether $f$ is convex or not \cite[Chapter 3]{boyd2004convex}. It can be readily seen from \eqref{dual opt} that $g$ is a concave function (since $f^*$ is convex, $-f^*$ is concave).  As shown in \cite[Section 3.5]{zalinescu2002convex}, strong-convexity of function $f$ and strong-smoothness of its conjugate $f^*$ are equivalent. Using this, the following results can be stated.

\begin{Lemma}\label{smooth convex}
If $f$ is a convex, $\beta$-strongly smooth function, then the function $g$ defined as per \eqref{g func} is $\alpha$-strongly concave, for some $\alpha>0$.
\end{Lemma}
\noindent The proof is provided in Appendix \ref{app smooth convex}. Thus, the following assumption on $f$ is made. 

\begin{Assumption}\label{assum f star diff}
The objective function $f\in C^{1,1}_{loc}(\mathbb R^n, \mathbb R)$ is $\beta_1$-strongly convex, $\beta_2$-strongly smooth, its conjugate function $f^*$ is known in closed-form, and satisfies $f^*\in C^{1,1}_{loc}(\mathbb R^n, \mathbb R)$.
\end{Assumption}

\noindent See Remark \ref{Rem last} and Corollary \ref{Cor opt lin const FxTS} for the case when $f^*$ is not known in closed-form. Consider the dynamical system
\begin{equation}\label{y p1 p2}
    \dot \nu =  -c_1\frac{-\nabla g(\nu)}{\|\nabla g(\nu)\|^\frac{p_1-2}{p_1-1}}-c_2\frac{-\nabla g(\nu)}{\|\nabla h(\nu)\|^\frac{p_2-2}{p_2-1}},
\end{equation}
with $c_1, c_2>0$, $p_1>2$ and $1<p_2<2$. Note that the assumptions on functions $f, f^*$, and matrix $A$ implies $\sup\limits_\nu \inf\limits_x L(x,\nu) = \inf\limits_x \sup\limits_\nu L(x,\nu)$ (\hspace{-0.1pt}\cite[Section 5.5]{boyd2004convex}). Also, using Proposition \ref{prop exist unique}, one has that the solutions of \eqref{y p1 p2} exist and are unique for all $\nu(0)\in \mathbb R^m$. 

\begin{Lemma}\label{ascent lemma}
The trajectories of \eqref{y p1 p2} reach the optimal point $\nu^{\star}$ of \eqref{max optim} in fixed time $T_\nu$ for all initial conditions $\nu(0)\in \mathbb R^m$. 
\end{Lemma}
\begin{proof}
Per Lemma \ref{smooth convex}, $-g(\nu)$ is $\alpha$-strongly convex. Thus, $g$ satisfies PL inequality \eqref{PL ineq} with some constant $\mu_g>0$. Since $g$ is strongly convex, and the maximizer $\nu^\star$ of $g$ exists, it is also unique. This implies that $g$ satisfies Assumptions \ref{assum exist xst} and \ref{f assum 2}. Hence, using Theorem \ref{Fixed time GF}, one obtains that the trajectories of \eqref{y p1 p2} reach the the maximizer $\nu^\star$ of $g(\nu)$ in a fixed time $T_\nu \leq \frac{4}{k_3(2-\alpha_1)}+\frac{4}{k_4(\alpha_2-2)} $ for all initial conditions $\nu(0)$, where $\alpha_1 = 2-\frac{p_1-2}{p_1-1}, \alpha_2 = 2-\frac{p_2-2}{p_2-1}, k_3 = c_12^{\frac{2+3\alpha_1}{4}}\mu_g^{\frac{\alpha_1}{2}}$ and $k_4 = c_22^{\frac{2+3\alpha_2}{4}}\mu_g^{\frac{\alpha_2}{2}}$. 
\end{proof}

Under the assumption of existence (Assumption \ref{assum A full rank}) and uniqueness (guaranteed by Assumption \ref{assum f star diff}) of the optimal point of \eqref{Equal cons} and using the fact that $-g(\nu)$ is $\alpha$-strongly convex, one obtains that the minimizer of $L(x,\nu^{\star})$ is the optimal solution of \eqref{Equal cons} \cite[Section 5.5.5]{boyd2004convex}. Using this, one obtains that 
\begin{equation}\label{lagr nu}
    x^\star = \textrm{arg}\min_{x\in \mathbb R^n} \; L(x,\nu^{\star})
\end{equation}
or, in other words, $x^\star$ satisfies $\nabla_x L(x^{\star}, \nu^{\star}) \triangleq \nabla f(x^{\star}) + \nu^{*T}A = 0$. Hence, the trajectories of the system 
\begin{equation}\label{x q1 q2}
    \dot x = -d_1\frac{\nabla_x L(x,\nu^{\star})}{\|\nabla_x L(x,\nu^{\star})\|^\frac{q_1-2}{q_1-1}}-d_2\frac{\nabla_x L(x,\nu^{\star})}{\|\nabla_x L(x,\nu^{\star})\|^\frac{q_2-2}{q_2-1}},
\end{equation}
with $d_1, d_2>0$, $q_1>2$ and $1<q_2<2$, converge to the optimizer of \eqref{Equal cons}. The following result can now be stated.

\begin{Theorem}
Let Assumptions \ref{assum A full rank} and \ref{assum f star diff} hold. Then, the optimal point $x^{\star}$ of \eqref{Equal cons} can be found in fixed time $T_{eq}$ by first solving \eqref{y p1 p2} for any $\nu(0)\in \mathbb R^m$, and then, solving \eqref{x q1 q2} for any $x(0)\in \mathbb R^n$, with
\begin{align*}
    T_{eq} \leq \frac{4}{k_3(2-\alpha_1)}+\frac{4}{k_4(\alpha_2-2)}+\frac{4}{k_5(2-\alpha_3)}+\frac{4}{k_6(\alpha_4-2)}, 
\end{align*}
where $k_i, \alpha_i$ are functions of $c_1, c_2, d_1, d_2, p_1, p_2, q_1, q_2$.
\end{Theorem}

\begin{proof}
From Lemma \ref{ascent lemma}, one obtains that the trajectories of \eqref{y p1 p2} reach the optimizer $\nu^{\star}$ of \eqref{max optim} in fixed time $T_\nu\leq \frac{4}{k_3(2-\alpha_1)}+\frac{4}{k_4(\alpha_2-2)}$. Now, since $f(x)$ is strongly convex, it follows that $L(\cdot,\nu)$ is strongly convex for each $\nu \in \mathbb R^m$, and in particular, $L(\cdot,\nu^{\star})$ is strongly convex, and hence, also satisfies PL inequality for some constant $\mu_L>0$. Furthermore, it can be easily shown that $L(\cdot, \nu^\star)$ satisfies Assumptions \ref{assum exist xst} and \ref{f assum 2}. Therefore, from Theorem \ref{Fixed time GF}, one has that there exists a fixed time $T_x$ such that the trajectories of \eqref{x q1 q2} reach the optimal point of \eqref{lagr nu} in $T_x \leq \frac{4}{k_5(2-\alpha_3)}+\frac{4}{k_6(\alpha_4-2)}$ for all initial conditions $x(0)$, where $\alpha_3 = 2-\frac{q_1-2}{q_1-1}, \alpha_4 = 2-\frac{q_2-2}{q_2-1}, k_5 = d_12^{\frac{2+3\alpha_3}{4}}\mu_L^{\frac{\alpha_1}{2}}$ and $k_6 = d_22^{\frac{2+3\alpha_4}{4}}\mu_L^{\frac{\alpha_4}{2}}$. Hence, one has that the optimal point of \eqref{Equal cons} can be obtained in fixed time $T_{eq} \leq T_x+T_\nu$, by first solving \eqref{y p1 p2} and then, \eqref{x q1 q2}.
\end{proof}
 
A very important class of optimization problems in machine learning and model predictive control (MPC) is the class of quadratic programs (QPs). In the following example, it is shown that QPs with equality constraints that satisfy Assumption \ref{assum f star diff} fit into the proposed framework. 
\begin{Example}\label{Example: QP dual}
Consider the following QP with equality constraints
\begin{equation}\label{opt quad lin cons ex}
    \begin{split}
        \min_{x\in \mathbb R^n} &\; \frac{1}{2}x^TQx+c^Tx, \\
        \textnormal{s.t.} \; &Ax = b,
    \end{split}
\end{equation}
where $Q\in \mathbb R^{n\times n}$ is positive definite and $A\in \mathbb R^{m\times n}$ has full row-rank. The function $g(\nu)$ for \eqref{opt quad lin cons ex} is given as
\begin{equation*}
\begin{split}
    g(\nu) &= \inf_{x\in \mathbb R^n}  \Big(\frac{1}{2}x^TQx+c^Tx+\nu^T(Ax-b)\Big)\\
    & = -\nu^Tb-\frac{1}{2}(c-A^T\nu)^TQ^{-1}(c-A^T\nu).
\end{split}
\end{equation*}
Hence, one has that $f^*(-A^T\nu) = -\frac{1}{2}(c-A^T\nu)^TQ^{-1}(c-A^T\nu)$. It can be readily verified that the functions $f, f^*$ satisfy Assumption \ref{assum f star diff}.  
\end{Example}

{
\begin{Remark}
The assumption on locally Lipschitz continuity of the gradient of $f$ or $f^*$ (in Assumptions \ref{f assum 2} and \ref{assum f star diff}) is sufficient for uniqueness of the solution of the concerned modified GF. This assumption can be further relaxed by using \cite[Theorem 3.15.11]{agarwal1993uniqueness}, where only continuity of the vector field is shown to be sufficient for existence and uniqueness of the solutions. Furthermore, using \cite[Theorem 3.5.10]{zalinescu2002convex}, one obtains that $f^*\in C^1$ if $f$ is strongly convex. Thus, Assumption \ref{assum f star diff} (as well as Assumption \ref{f assum 2}) can be relaxed by allowing $f\in C^1$ and imposing no additional assumptions on $f^*$.
\end{Remark}
}

\section{FxTS of Saddle-Point Dynamics}\label{FT SP}

In this section, min-max problems are considered that can be formulated as saddle-point dynamics, and a modification is studied so that optimal point, which is a saddle-point, can be found within a fixed time.  To this end, the general saddle-point problem of minimizing a function $F(x,z)$ over $x\in\mathbb R^n$ and maximizing over $z\in\mathbb R^m$ is considered, where $F:\mathbb R^n\times\mathbb R^m\rightarrow\mathbb R$. Formally, this can be stated as 
\begin{equation}\label{min max problem}
    \max_{z\in \mathbb R^m}\min_{x\in \mathbb R^n} F(x,z).
\end{equation}
A point $(x^\star,z^\star)$ is called as \textit{local saddle-point} of $F$ (as well as local optimal solution of \eqref{min max problem}), if there exist open neighborhoods $U_x\subset \mathbb R^n$ and $U_z\subset \mathbb R^m$ of $x^\star$ and $z^\star$, respectively, such that for all $(x,z)\in U_x\times U_z$, one has
\begin{equation}
    F(x^\star,z) \leq F(x^\star,z^\star) \leq F(x,z^\star).
\end{equation}
The point $(x^\star,z^\star)$ is \textit{global saddle-point} if $U_x = \mathbb R^n$ and $U_z = \mathbb R^m$. The following assumption is made.

\begin{Assumption}\label{F assum 1}
A saddle point $(x^\star, z^\star)$ exists that solves \eqref{min max problem}. Furthermore, the function $F\in C^2(\mathbb R^n\times\mathbb R^m, \mathbb R)$ is locally strictly convex-concave in an open neighborhood $U \subset \mathbb R^n\times \mathbb R^m$ of the saddle point $(x^\star,z^\star)$. More specifically, $\nabla_{xx} F(x,z) \succ 0$ and $\nabla_{zz} F(x,z)\prec 0$ for all $(x,z)\in U$.
\end{Assumption}

The local strong or strict convexity-concavity assumption is very commonly used in literature for showing asymptotic convergence of saddle-point dynamics to the optimal solution of \eqref{min max problem} (see, e.g., \cite{cherukuri2017saddle,cherukuri2017role}). Using this, the following result can be stated. 

\begin{Lemma}\label{M inv}
Let Assumption \ref{F assum 1} hold for some open neighborhood $U\subset\mathbb R^n\times\mathbb R^m$ of the saddle-point $(x^\star,z^\star)$. Then, the Hessian of $F$ given as
\begin{equation}
     \nabla^2 F(x,z) = \begin{bmatrix} \nabla_{xx} F(x,z) & \nabla_{xz} F(x,z)\\ \nabla_{zx} F(x,z) & \nabla_{zz} F(x,z)\end{bmatrix},
\end{equation}
is invertible for all $(x,z)\in U$.
\end{Lemma}
\noindent The proof is provided in Appendix \ref{app proof M inv}. Authors in \cite{cherukuri2017saddle} use the following saddle-point dynamics
\begin{equation}\label{SP asym dyn}
    \dot x = -\nabla F_x(x,z), \quad
    \dot z  = \nabla F_z(x,z).
\end{equation}
and show asymptotic convergence to the saddle-point $(x^\star, z^\star)$ under Assumption \ref{F assum 1}. Next, the flow in \eqref{SP asym dyn} is modified so that fixed-time convergence can be guaranteed. The FxTS Newton's method is used to define the FxTS saddle-point (FxTS-SP) dynamics as{\small
\begin{equation}\label{SP dyn}
\begin{split}
    \begin{bmatrix}\dot x \\ \dot z\end{bmatrix} & = -(\nabla^2 F(x,z))^{-1}\left(c_1\frac{\nabla F(x,z)}{\|\nabla F(x,z)\|^\frac{p_1-2}{p_1-1}}\right.\\
    & \quad +\left.c_2\frac{\nabla F(x,z)}{\|\nabla F(x,z)\|^\frac{p_2-2}{p_2-1}}\right),
\end{split}
\end{equation}}\normalsize
where $c_1, c_2>0$, $p_1>2, 1<p_2<2$, and $\nabla F(x,z) \triangleq \begin{bmatrix}\nabla_x F(x,z)^T & \nabla_z F(x,z)^T\end{bmatrix}^T$. Note that per Lemma \ref{eq point fixed point}, the point $(x,z)$ is an equilibrium point of \eqref{SP dyn} if and only if it satisfies $\nabla F(x,z) = 0$. Using strict convexity-concavity of $F$ in $U$, one obtains that $\nabla F(x,z) = 0$ implies $x = x^\star$ and $z = z^\star$. The first main result of this section is presented below. 

\begin{Theorem}
If $F$ satisfies Assumption \ref{F assum 1} for some $U\subset \mathbb R^n\times \mathbb R^m$, then the trajectories of \eqref{SP dyn} converge to the saddle point $(x^\star, z^\star)$ in fixed time $T_{SP}$ for all $(x(0),z(0))\in D\subset U$ where $D$ is the largest compact sub-level set of $V(x,z) = \frac{1}{2}\|\nabla F(x,z)\|^2$ in $U$. Furthermore, if $U = \mathbb R^n\times \mathbb R^m$, then the results holds for all $(x(0), z(0))\in \mathbb R^n\times \mathbb R^m$. 
\end{Theorem}
\begin{proof}
Consider the candidate Lyapunov function $V(x,z) = \frac{1}{2}\|\nabla F(x,z)\|^2$. Define $D$ as the largest compact sub-level set of $V$. Using analysis similar to the proof of Theorem \ref{NT FT}, the time derivative of $V$ along the trajectories of \eqref{SP dyn} can be bounded as
\begin{equation*}
    \dot V\leq -c_12^{\frac{\alpha_1}{2}} V^{\frac{\alpha_1}{2}}-c_22^{\frac{\alpha_2}{2}} V^{\frac{\alpha_2}{2}},
\end{equation*}
where $\alpha_1 = 2-\frac{p_1-2}{p_1-1}$ and $\alpha_2 =2- \frac{p_2-2}{p_2-1}$. It follows that for all $t\geq T_{SP}$, $\nabla F(x(t), z(t)) = 0$, or equivalently,  $(x(t), z(t)) = (x^\star, z^\star)$, where $T_{SP} \leq \frac{2^{1-\frac{\alpha_1}{2}}}{c_1(2-\alpha_1)} +  \frac{2^{1-\frac{\alpha_2}{2}}}{c_2(\alpha_2-2)}$ for all $(x(0),z(0))\in D$. 

For the case when $U = \mathbb R^n\times\mathbb R^m$, the sub-level sets of $\|\nabla F\|$ are bounded and so, $V$ is radially unbounded. Therefore, the trajectories of \eqref{SP dyn} converge to the saddle-point of \eqref{min max problem} for all $(x(0), z(0))\in \mathbb R^n\times \mathbb R^m$.
\end{proof}

Assumption \ref{F assum 1} ensures that the Hessian $\nabla^2 F(x,z)$ is invertible for all $(x,z)\in U$, and that the the saddle-point of $F$ is the only critical point. Per the analysis in Lemma \ref{M inv}, a sufficient condition for the Hessian to be invertible is that $\nabla_{xx} F$ is invertible and $\nabla_{zx} F$ is full row-rank. On the basis of this observation, the following result can be stated.

\begin{Cor}\label{cor 1}
Suppose there exists an open set $U\subset \mathbb R^n\times\mathbb R^m$ such that $\nabla_{xx} F(x,z)$ is invertible and $\nabla_{zx} F(x,z)$ is full row-rank for all $(x,z)\in U$. Then, the trajectories of \eqref{SP dyn} converge to the set of the critical points of $F$, defined as $\Omega_U = \{(x,z) \in U\; | \; \nabla F(x,z) = 0\}$ in a fixed time $T_{SP}$ for all $(x(0),z(0))\in D\subset U$, where $D$ is the largest compact sub-level set of $V = \frac{1}{2}\|\nabla F\|^2$ in $U$.
\end{Cor}

\begin{figure*}[hbt!]
\setcounter{equation}{29}
\small
\begin{align}\label{dot V SP first}
    \dot V & = \nabla_x F^T\nabla_{xx}F\dot x +\nabla_x F^T\nabla_{xz}F\dot z+\nabla_z F^T\nabla_{zx}F\dot x+\nabla_z F^T\nabla_{zz}F\dot z \overset{\eqref{SP first order}}{=}\nabla_x F^T\nabla_{xx}F\dot x+\nabla_z F^T\nabla_{zz}F\dot z\nonumber
    \\
    & =-\nabla_x F^T\nabla_{xx}F\left(c_1\frac{\nabla F_x(x,z)}{\|\nabla F(x,z)\|^\frac{p_1-2}{p_1-1}}+c_2\frac{\nabla F_x(x,z)}{\|\nabla F(x,z)\|^\frac{p_2-2}{p_2-1}}\right)+ \nabla_z F^T\nabla_{zz}F\left(c_1\frac{\nabla F_z(x,z)}{\|\nabla F(x,z)\|^\frac{p_1-2}{p_1-1}}+c_2\frac{\nabla F_z(x,z)}{\|\nabla F(x,z)\|^\frac{p_2-2}{p_2-1}}\right).
\end{align}
\setcounter{equation}{28}
\hrulefill
\end{figure*}

Note that Corollary \ref{cor 1} does not require strict convexity-concavity of $F$. Also, if the set $\Omega_U$ contains only the saddle-point, i.e., the only critical point of the function $F$ in $\Omega_U$ is the saddle-point, then Corollary \ref{cor 1} guarantees local convergence of trajectories of \eqref{SP dyn} to the saddle-point in fixed time. The final remark to connect the results in Corollary \ref{cor 1} with those in Section \ref{Sec Eq Const} are made below. 

\begin{Remark}\label{Rem last}
For problem \eqref{Equal cons} with strictly convex $f$ and full row-rank matrix $A$, the conditions of Corollary \ref{cor 1} are satisfied. Furthermore, it is known that for this problem, the Karush-Kuhn-Tucker (KKT) conditions are also sufficient for optimality, i.e., the critical point $(\bar x,\bar z)$ such that $\nabla F(\bar x,\bar z) \triangleq \begin{bmatrix}\nabla f(\bar x) + A^T\bar z\\ A\bar x-b\end{bmatrix} = 0$ is also the optimal point, i.e., $(x^\star, z^\star) = (\bar x, \bar z)$. Hence, one can use \eqref{SP dyn} with $F(x,z) = L(x,z)$ for the problems of the form 
\eqref{Equal cons}, in the case when the conjugate function $f^*$ is not known in the closed-form. 
\end{Remark}

\noindent This is formally shown in the following result.

\begin{Cor}\label{Cor opt lin const FxTS}
Consider the optimization problem \eqref{Equal cons}. Assume that Assumption \ref{assum A full rank} holds and that $f\in C^2(\mathbb R^n, \mathbb R)$ is strictly convex. Then, with $F(x, z) = f(x) + z^T(Ax-b)$, the trajectories of \eqref{SP dyn} reach the saddle-point $(x^\star, z^\star)$, where $x^\star$ is the solution of \eqref{Equal cons}, in fixed time $T_{eq2}<\infty$ for all $(x(0), z(0))\in D\subset U$, where $D$ is the largest compact sub-level set of $V = \frac{1}{2}\|\nabla F\|^2$ in $U$. 
\end{Cor}
\begin{proof}
Define $F(x, z) = f(x) + z^T(Ax-b)$. Note that for strictly convex $f$, $\nabla_{xx}F = \nabla^2 f(x)$ is invertible. Furthermore, $\nabla_{zx}F = A$ is full row-rank, which implies that the conditions of Corollary \ref{cor 1} are satisfied. Hence, the trajectories of \eqref{SP dyn} reach the set of points $(x,z)$ such that $\nabla F(x, z) = 0$. Hence, one obtains that the trajectories of \eqref{SP dyn} for $F(x, z) = f(x) + z^T(Ax-b)$ reach the optimal point of \eqref{Equal cons} in fixed time $T_{eq2} \leq \frac{2^{1-\frac{\alpha_1}{2}}}{c_1(2-\alpha_1)} +  \frac{2^{1-\frac{\alpha_2}{2}}}{c_2(\alpha_2-2)}$. 
\end{proof}

In summary, \eqref{SP dyn} can be used to solve constrained optimization problems of the form \eqref{Equal cons} as well as min-max problems of the form \eqref{min max problem}, and the optimal solutions can be obtained within a fixed time. Compared to \cite{cherukuri2017role,cherukuri2017saddle}, where asymptotic convergence is studied for min-max problems of the form \eqref{min max problem}, and for \eqref{Equal cons} posed as saddle-point problem, the proposed method guarantees convergence within a fixed time under relaxed assumptions. 

The proposed method \eqref{SP dyn} requires computation of the inverse of the Hessian matrix $\nabla^2 F(x,z)$, which can be computationally expensive for problems with large $n,m$. Next, a first-order e, i.e., a method only requiring the gradient of the function $F$, is proposed under a stronger assumption on the function $F$.

\begin{Assumption}\label{F assum 2}
A saddle point $(x^\star, z^\star)$ exists that solves \eqref{min max problem} and $F\in C^2(\mathbb R^n\times\mathbb R^m, \mathbb R)$ is locally strongly convex-concave on open neighborhood $U \subset \mathbb R^n\times \mathbb R^m$ of the saddle point $(x^\star,z^\star)$, i.e., there exist $k_x, k_z>0$ such that $\nabla_{xx}  F(x,z)\succeq k_xI$ and $\nabla_{zz} F(x,z)\preceq -k_zI$ for all $(x,z)\in U$.
\end{Assumption}

Consider the following FxTS-GF based modified saddle-point dynamics
\begin{equation}\label{SP first order}
\begin{split}
    \begin{bmatrix}\dot x\\\dot z\end{bmatrix} &= -c_1\frac{\tilde \nabla F(x,z)}{\|\nabla F(x,z)\|^\frac{p_1-2}{p_1-1}}-c_2\frac{\tilde \nabla F(x,z)}{\|\nabla F(x,z)\|^\frac{p_2-2}{p_2-1}},\\
\end{split}
\end{equation}
where $c_1, c_2>0$, $p_1>2, 1<p_2<2$, $\tilde \nabla F(x,z) \triangleq \begin{bmatrix}\nabla_x F(x,z)^T & -\nabla_z F(x,z)^T\end{bmatrix}^T$. Note that \eqref{SP asym dyn} is a special case of \eqref{SP first order} with $c_1 = 1, c_2 = 0$ and $p_1 = 2$. The following result can be readily stated for \eqref{SP first order}.

\begin{Theorem}\label{SP first order TH}
Suppose the function $F$ satisfies Assumption \ref{F assum 2}. Then, the trajectories of \eqref{SP first order} converge to the saddle-point in a fixed time $T_{SP2}$ for all $(x(0), z(0))\in U$. If $U = \mathbb R^n\times\mathbb R^n$, then the result holds for all $(x(0), z(0))\in\mathbb R^n\times\mathbb R^m$. 
\end{Theorem}

\begin{proof}
Choose the candidate Lyapunov function as $V(x,z) = \frac{1}{2}\|\nabla F(x,z)\|^2$. The time derivative of $V$ along the trajectories of \eqref{SP first order} reads as in \eqref{dot V SP first}. Now, using the strong convexity-concavity of $F$, one obtains {\small
\begin{align*}
    \dot V& \leq -c_1k_x\frac{\|\nabla_x F\|^2}{\|\nabla F\|^\frac{p_1-2}{p_1-1}}-c_2k_x\frac{\|\nabla_x F\|^2}{\|\nabla F\|^\frac{p_1-2}{p_1-1}}\\
    & -c_1k_z\frac{\|\nabla_z F\|^2}{\|\nabla F\|^\frac{p_1-2}{p_1-1}}-c_2k_z\frac{\|\nabla_z F\|^2}{\|\nabla F\|^\frac{p_1-2}{p_1-1}}\\
    & \leq  -c_1k\frac{\|\nabla F\|^2}{\|\nabla F\|^\frac{p_1-2}{p_1-1}}-c_2k\frac{\|\nabla F\|^2}{\|\nabla F\|^\frac{p_1-2}{p_1-1}} = -k_7V^{\frac{\alpha_7}{2}}-k_8V^{\frac{\alpha_8}{2}},
\end{align*}}\normalsize
where $k_7 = c_1k2^{\frac{\alpha_7}{2}},k_8 = c_2k2^{\frac{\alpha_8}{2}}$, $0<\alpha_7 = 2-\frac{p_1-2}{p_1-1}<2$ and $\alpha_8 = 2-\frac{p_2-2}{p_2-1}>2$, where $k = \min\{k_x, k_z\}$. Hence, using Lemma \ref{FxTS TH}, one obtains that the optimal point of \eqref{min max problem} can be found in fixed time $T_{SP2}$ satisfying $T_{SP2}\leq \frac{2}{k_7(2-\alpha_7)}+\frac{2}{k_8(\alpha_8-2)}$. Furthermore, the norm of the gradient $\|\nabla F\|$ is radially unbounded on $U$, and hence, for $U = \mathbb R^n\times\mathbb R^m$, the result holds globally for any $(x(0), z(0))$. 
\end{proof}

\section{Numerical Examples}\label{simulations}
The efficacy of the proposed methods is illustrated via three numerical examples.  The computations are done using MATLAB R2018a on a desktop with a 32GB DDR3 RAM and an Intel Xeon E3-1245 processor (3.4 GHz). Unless mentioned otherwise, Euler discretization is used for \textsc{Matlab} implementation with time-step $dt=10^{-5}$, { and with constant step-size, the convergence time $T$ in seconds translates to $T\times 10^5$ iterations.} 
In the first example, an instance of the logistic regression based support-vector machine is considered, where the performance of the proposed FxTS-GF is compared with Newton's method. The flow in \eqref{fixed flow} is used to find the optimal solution, i.e., the separating hyperplane for a given labelled data set, within a fixed time. In the second example, an instance of QP with equality constraints is considered as a constrained convex optimization problem \eqref{Equal cons}. The FxTS saddle-point dynamics in \eqref{SP dyn} is used to find the optimal point of the problem, and to illustrate that for any initial condition, the optimal point can be found within a fixed time. 
Finally, an instance of the min-max problem \eqref{min max problem} is considered, and the FxTS saddle-point dynamics in \eqref{SP dyn} is used to find the saddle-point. 

\subsection{Support Vector Machine: Unconstrained optimization}
Consider an instance of logistic SVM, where the function $f:\mathbb R^2\rightarrow \mathbb R$ is given as
\begin{align}
    f(x) = \frac{1}{2}\|x\|^2+\frac{1}{\mu}\sum_{i = 1}^N\log\left(1+\exp(-\mu l_{i}x^Tz_{i})\right),
\end{align}
with $\mu>0$ a large positive number. Here $x = x^\star\in \mathbb R^2$ represents the separating hyperplane, and $z_{i}\in \mathbb R^2$ and $l_{i}\in \{-1, 1\}$ denote the $i$-th data point and its corresponding label, respectively. The vectors $z_{i}$ are chosen from a random distribution around the line $x_1 = x_2$, so that the solution $x^\star$, i.e., the separating hyperplane, to the minimization problem $\min\limits_{x\in \mathbb R^n} f(x)$ is the vector $[1, \; -1]$. In this case, $N = 500$ randomly distributed data points are considered. {The parameters used in the numerical simulations are $c_1 = c_2 = 10$, $p_1 = 2.6, p_2 = 1.6$, and $\mu = 2$. The theoretical bound on the time of convergence is $T_1\leq 0.43$.} Figure \ref{fig:z dist} shows the distribution of $z_{i}$ around the line $x_1 = x_2$. 

\begin{figure}[!ht]
    \centering
        \includegraphics[ width=0.5\columnwidth,clip]{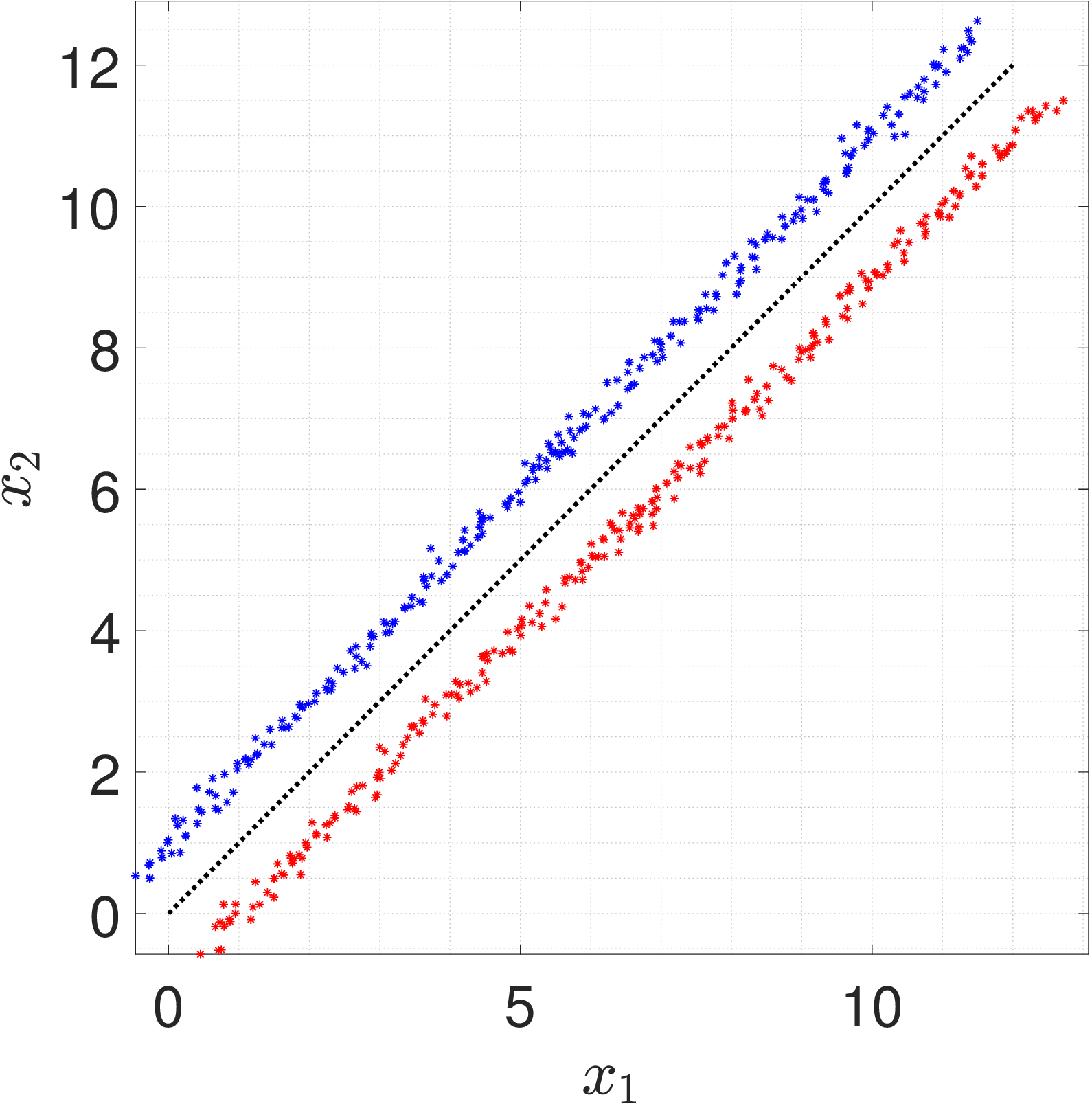}
        \caption{Distribution of points $z_{i}$ around the line $x_1 = x_2$ (red dotted line). Blue and red stars denote the points corresponding to $l_{i} = -1$ and $l_{i} = 1$, respectively. .}\label{fig:z dist}
\end{figure}

\begin{figure}[!ht]
    \centering
            \includegraphics[ width=1\columnwidth,clip]{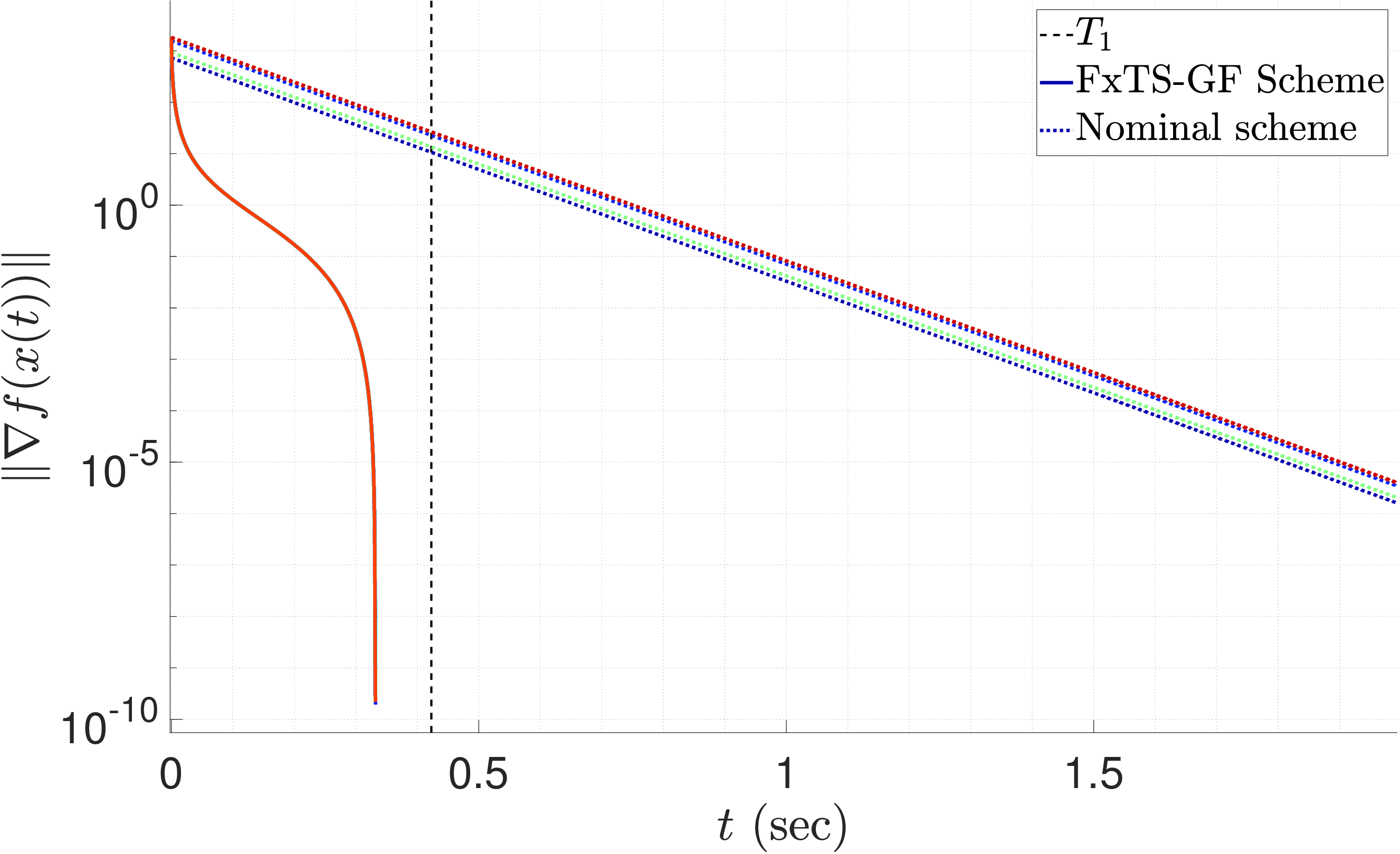}
        \caption{The norm of gradient, $\|\nabla f(x(t))\|$, with time for various initial conditions for the FxTS-GF (solid lines) and the Newton's method (dotted lines).}\label{fig: comparison SVM nabla}
\end{figure}

Figure \ref{fig: comparison SVM nabla} depicts the norm of the gradient $\|\nabla f(x(t))\|$ with time for various initial conditions. The $\log$ scale is used on $y-$ axis so that the variation of the norm $\|\nabla f\|$ is clearly shown for values near zero, and the super-linear nature of convergence can be demonstrated. Note that the plots corresponding to the nominal GF method are linear on the $\log$-scale, which verifies that the convergence is exponential, while the curved plots of the proposed scheme illustrate the super-linear convergence. It can also be noted that the convergence time (upto the error of $10^{-10}$) is bounded by the theoretical bound of $T_1$. Thus, per Figure \ref{fig: comparison SVM nabla}, nominal GF takes at least 5 times more iterations as compared to the FxTS-GF, in order to converge to the same level of accuracy.

\subsection{Example 2: QP with equality constraints}
Consider \eqref{opt quad lin cons ex} with $x\in \mathbb R^{10}$ and $A\in \mathbb R^{5\times 10}$. For simplicity, consider a diagonal matrix $Q$ with positive diagonal elements and a full row-rank matrix $A$, so that all the conditions of Corollary \ref{Cor opt lin const FxTS} are satisfied. The values of $Q, A, b, c$ are chosen through random matrix generator in \textsc{Matlab}. The following parameters are used for FxTS-SP dynamics in \eqref{SP dyn}: $c_1 = 10,\;  c_2 = 10,\; p_1 = 2.2, \; p_2 = 1.8$. With these parameters, the upper bound on the time of convergence in Corollary \ref{Cor opt lin const FxTS} satisfies $T_{eq2} = T_{SP}\leq 1.0025$. 

Figure \ref{fig:nom x QP} 
compares the performance of the proposed method relative to Newton's method for saddle-point dynamics, i.e., \eqref{SP dyn} with $c_2 = 0$ and $p_2 = 2$. The dotted lines illustrate the evolution of Newton's method, while solid lines illustrate that of FxTS-SP dynamics \eqref{SP dyn}. The vertical black dashed black line corresponds to $T_{SP} = 1.0025$ sec. Figure \ref{fig:nom x QP} shows the variation of $\|x(t)-x^\star\|$ with time for various initial conditions. The proposed scheme converges to the error of magnitude less than $10^{-8}$ within $T_{SP}$ sec, while the nominal scheme takes longer time (and thus, more number of iterations) to achieve the same. It can also be seen that the convergence time is always bounded by $T_{SP}$ for all initial conditions for the proposed method. 


\begin{figure}[!ht]
    \centering
        \includegraphics[ width=1\columnwidth,clip]{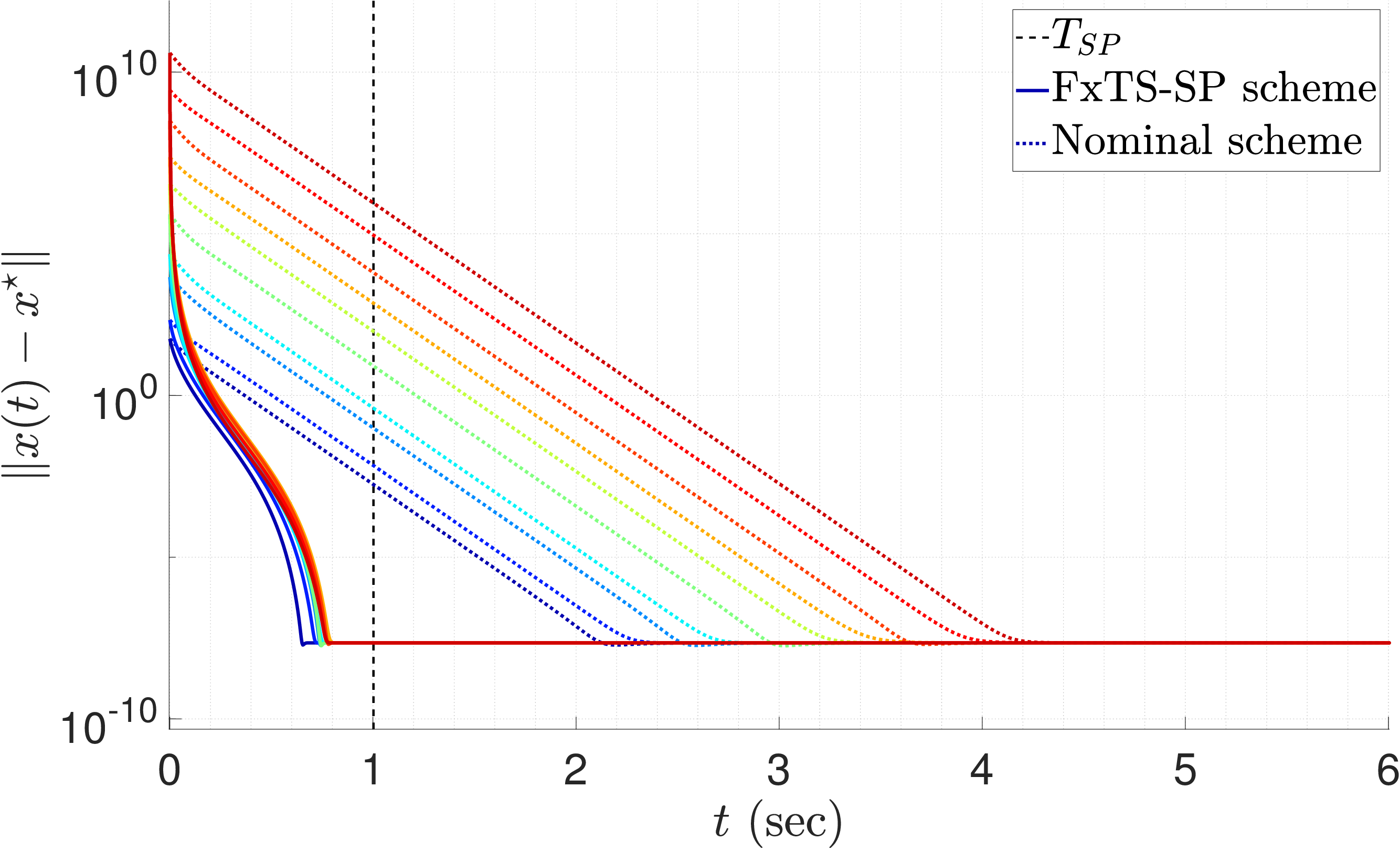}
        \caption{The norm $\|x(t)-x^\star\|$ with time for various initial conditions for nominal saddle-point dynamics ($p_1 = 2, c_2 = 0$) and FxTS saddle-point dynamics ($p_1= 2.2, p_2 = 1.8$). }
    \label{fig:nom x QP}
\end{figure}


\subsection{Example 3: Min-max problem}

A numerical example for the min-max problem $\max\limits_z\min\limits_xF(x,z)$ is considered, where the function $F$ is defined as:
\begin{equation}\label{eq: min max sim example}
    F(x,z) = (\|x\|-1)^4-\|z\|^2\|x\|^2,
\end{equation}
with $x\in \mathbb R^n$ and $z\in \mathbb R^m$. The dimensions are chosen as $n = 3$ and $m = 1$. The set of optimal points $(x,z)$ satisfy $\|x\| = 1, \|z\| = 0$  \cite{cherukuri2017saddle}, i.e., the optimal point is not unique in this case. The parameters $c_1$, $c_2$ are chosen as $c_1 = c_2 = 10$. 

The first case study considers a varying range of initial conditions $(x(0), z(0))$ and constant values of the parameters $p_1, p_2$, chosen as $p_1 = 2.2$ and $p_2 = 1.8$. Figure \ref{fig:tc e0} shows the convergence time (upto an error of $\|\nabla F(x,z)\|\leq 10^{-15})$ for various initial conditions $x(0), z(0)$. The results illustrate that the time of convergence does not depend upon the initial distance from the saddle point. Also, the actual time of convergence $T_c$ is lower than the upper bound $T_{SP}$.  

\begin{figure}[!ht]
    \centering
        \includegraphics[ width=1\columnwidth,clip]{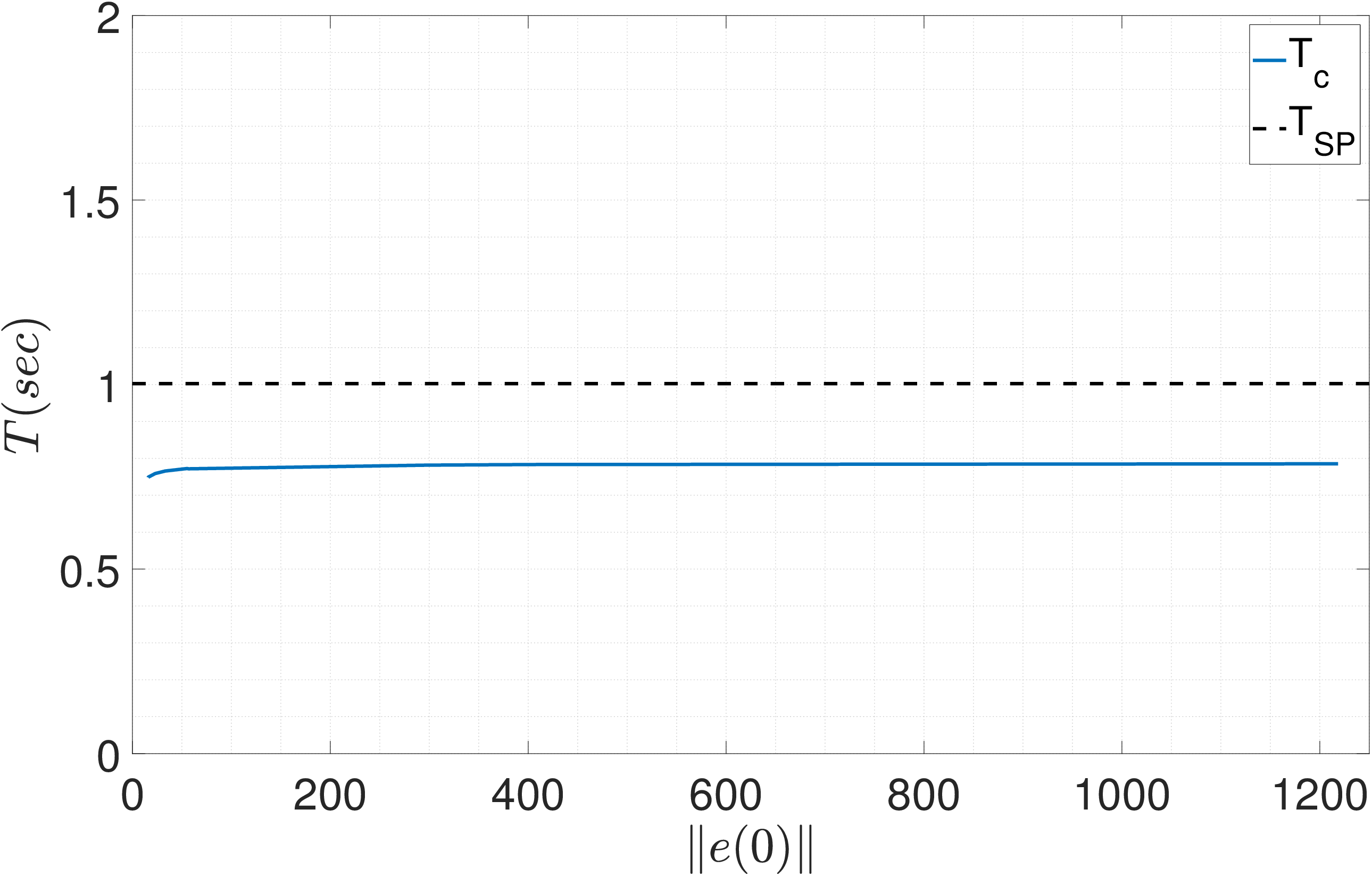}
        \caption{Time of convergence $T_c$ with norm of the initial error $\|e(0)\| \triangleq \|[(x(0)-x^\star)^T \; (z(0)-z^\star)^T]^T\| $.}
    \label{fig:tc e0}   
\end{figure}

\begin{figure}[!ht]
    \centering
        \includegraphics[ width=1\columnwidth,clip]{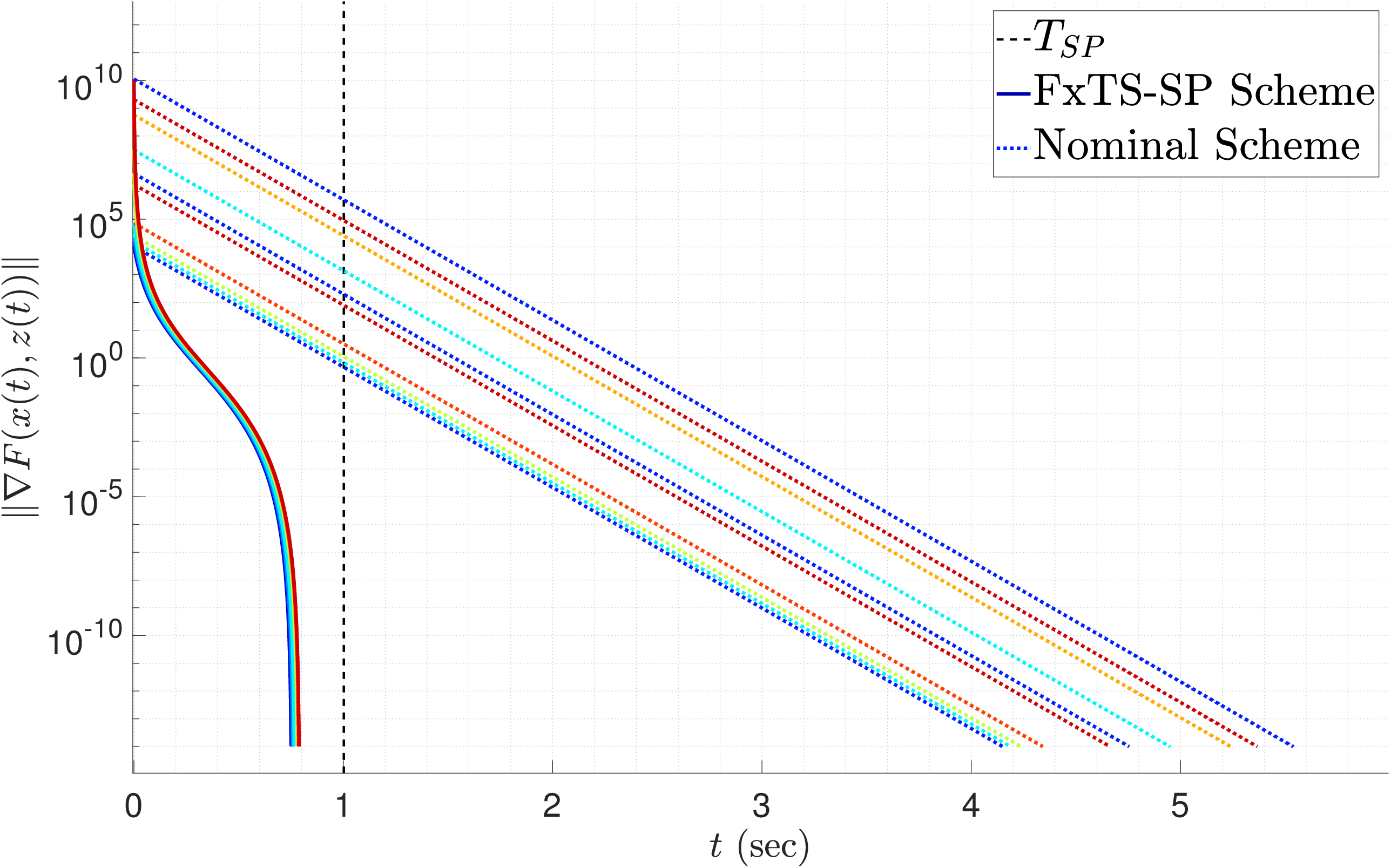}
        \caption{The norm of the gradient, $\|\nabla F(x(t),z(t))\|$, with time for various initial conditions for nominal saddle-point dynamics ($p_1 = 2, c_2 = 0$) and FxTS saddle-point dynamics ($p_1= 2.2, p_2 = 1.8$). }
    \label{fig:SP nabla F norm}
\end{figure}

Figure \ref{fig:SP nabla F norm} illustrates the convergence of norm of the gradient, $\|\nabla F(x,z)\|$, to zero in fixed time for various initial conditions. Figure \ref{fig:SP x norm} and \ref{fig:SP z norm} plot the norm of the error $x-x^\star$ and $z - z^\star$, respectively, for various initial conditions. Solid lines show the performance of the proposed method \eqref{SP dyn}, and dotted lines show the performance of Newton's method ($c_2 = 0, p_2= 2$) when solving for saddle-point dynamics. 
Again, it can be noted from the figures that the proposed method converges within the theoretical bound on $T_{SP}$, and has super-linear convergence rate. 

\begin{figure}[!ht]
    \centering
        \includegraphics[ width=1\columnwidth,clip]{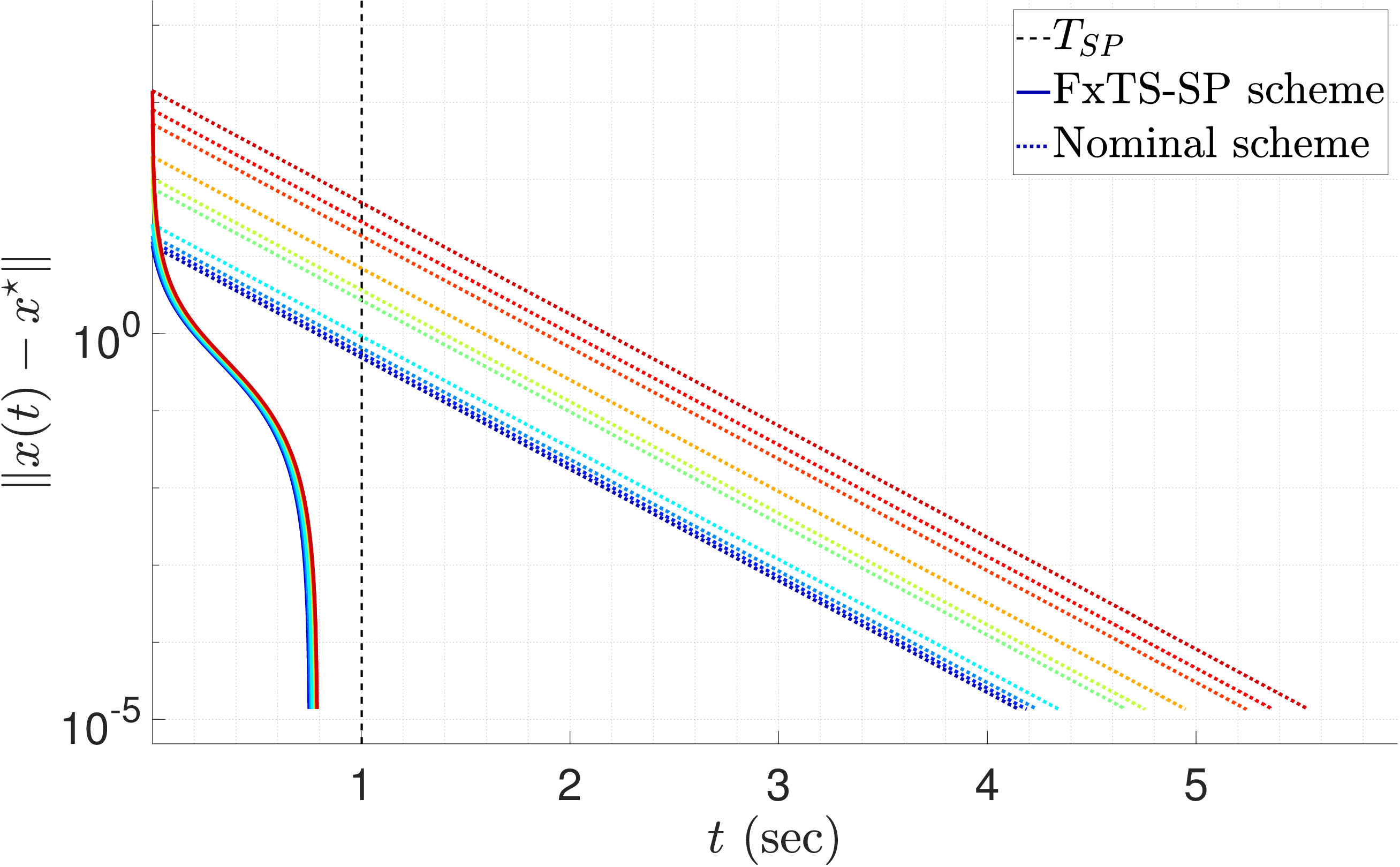}
        \caption{The norm $\|x-x^\star\|$ with time for various initial conditions for nominal saddle-point dynamics ($p_1 = 2, c_2 = 0$) and FxTS saddle-point dynamics ($p_1= 2.2, p_2 = 1.8$). }
    \label{fig:SP x norm}
\end{figure}

\begin{figure}[!ht]
    \centering
        \includegraphics[ width=1\columnwidth,clip]{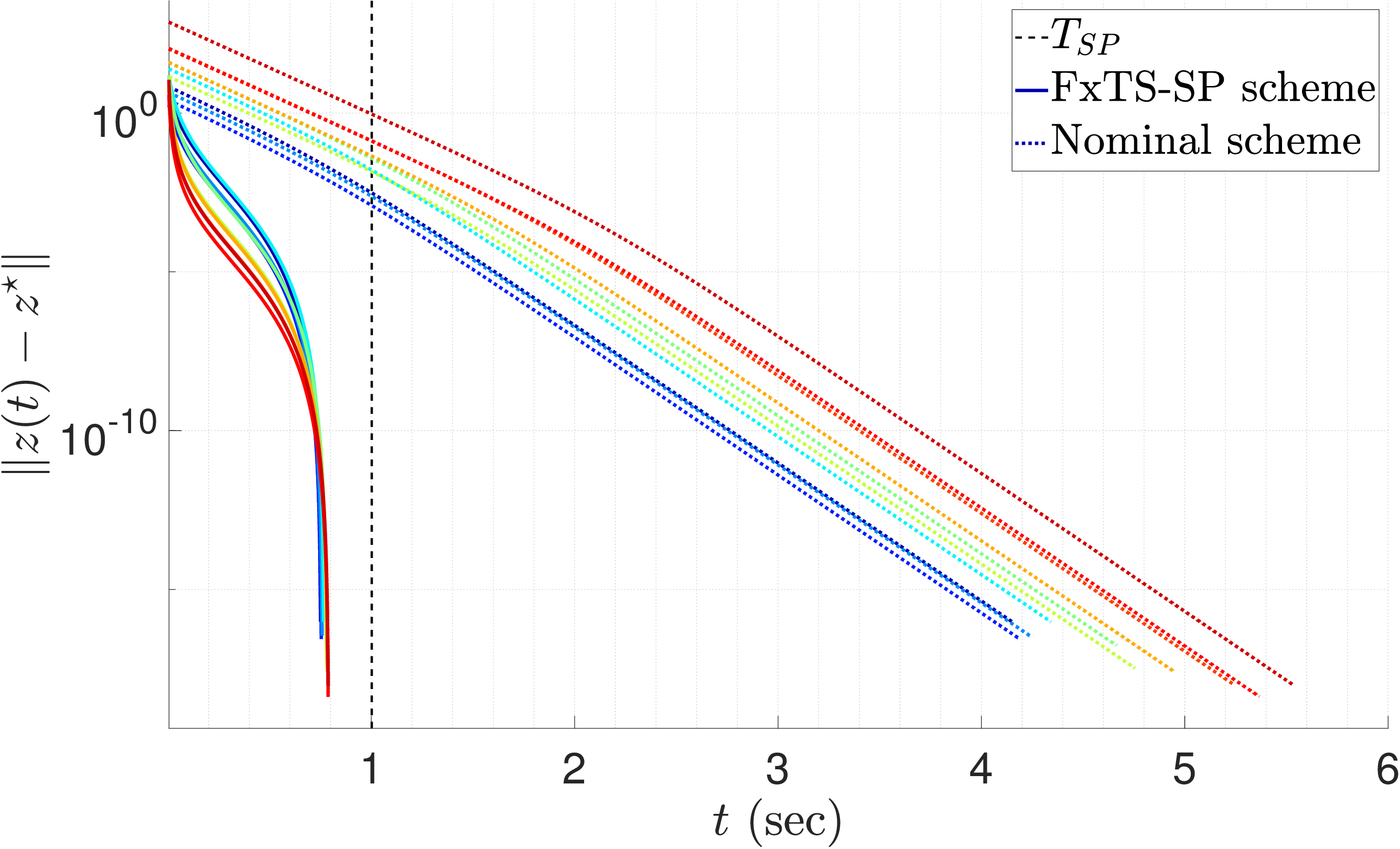}
        \caption{The norm $\|z-z^\star\|$ with time for various initial conditions for nominal saddle-point dynamics ($p_1 = 2, c_2 = 0$) and FxTS saddle-point dynamics ($p_1= 2.2, p_2 = 1.8$). }
    \label{fig:SP z norm}
\end{figure}

The second case study considers that the parameters $p_1, p_2$ are varied in the ranges $[2,\;2.2]$ and $[1.8, \;2]$, respectively. Figure \ref{fig:nabla f val} shows the norm of the gradient, $\|\nabla F(x,z)\|$, with time. As can be seen in the Figure \ref{fig:nabla f val}, the case when $p_1= p_2 = 2$ has linear convergence (straight line on the $\log$ plot), while for $p_1>2$ and $p_2<2$, the convergence is super-linear. It can also be observed that as $p_1$ increases and $p_2$ decreases, the convergence becomes faster and the time of convergence becomes smaller. 

\begin{figure}[!ht]
    \centering
        \includegraphics[ width=1\columnwidth,clip]{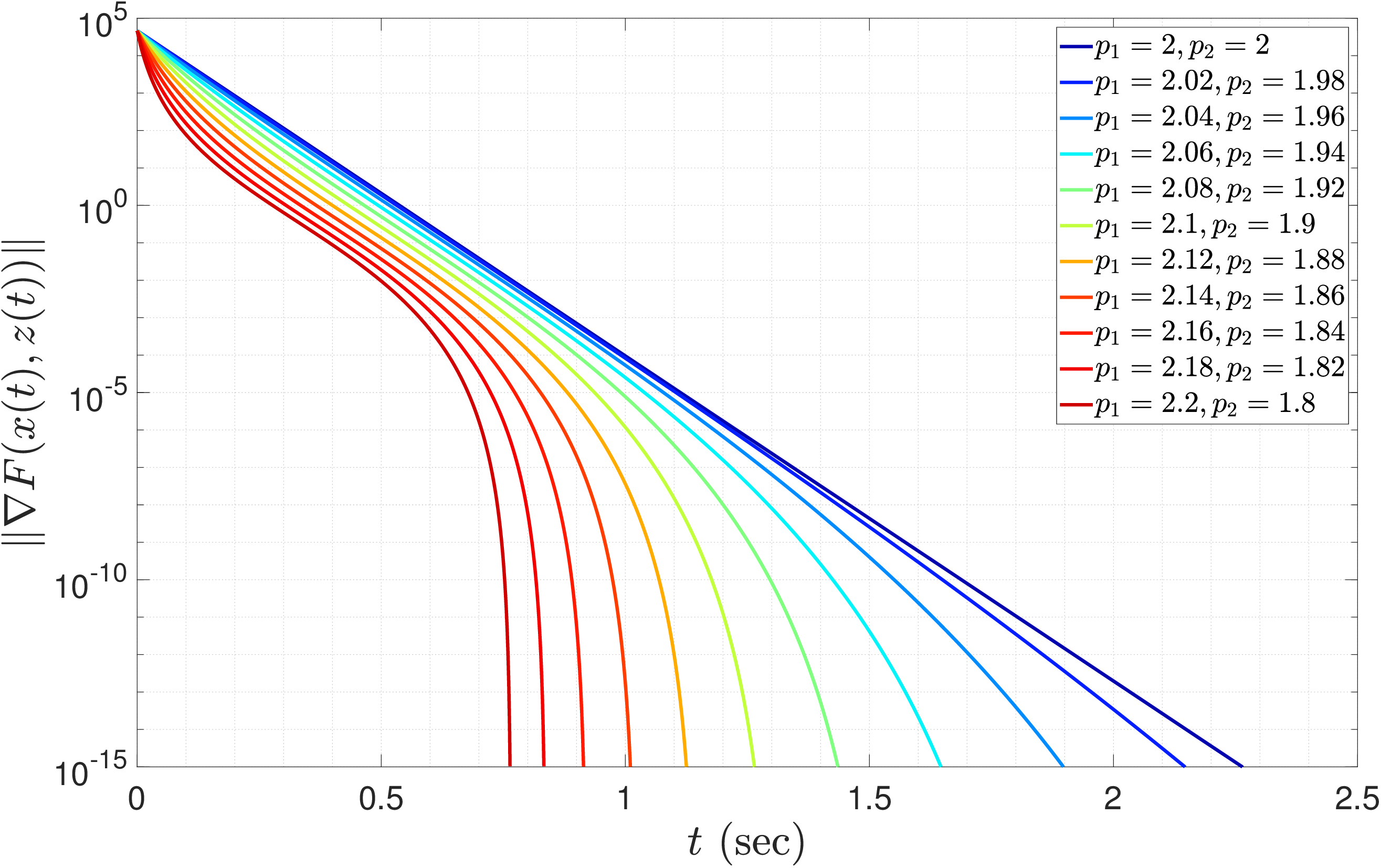}
        \caption{The norm of the gradient, $\|\nabla F(x(t),z(t))\|$, with time for various $p_1, p_2$.}
    \label{fig:nabla f val}
\end{figure}

The implementation of the proposed method in numerical studies is done using Euler integration with constant step size. Figure \ref{fig:nabla f dt} shows the performance of the proposed method for various values of discretization steps between $10^{-2}$ and $10^{-6}$. As the figure suggests, the discretization step does not affect the convergence performance of the proposed method.

\begin{figure}[!ht]
    \centering
        \includegraphics[ width=1\columnwidth,clip]{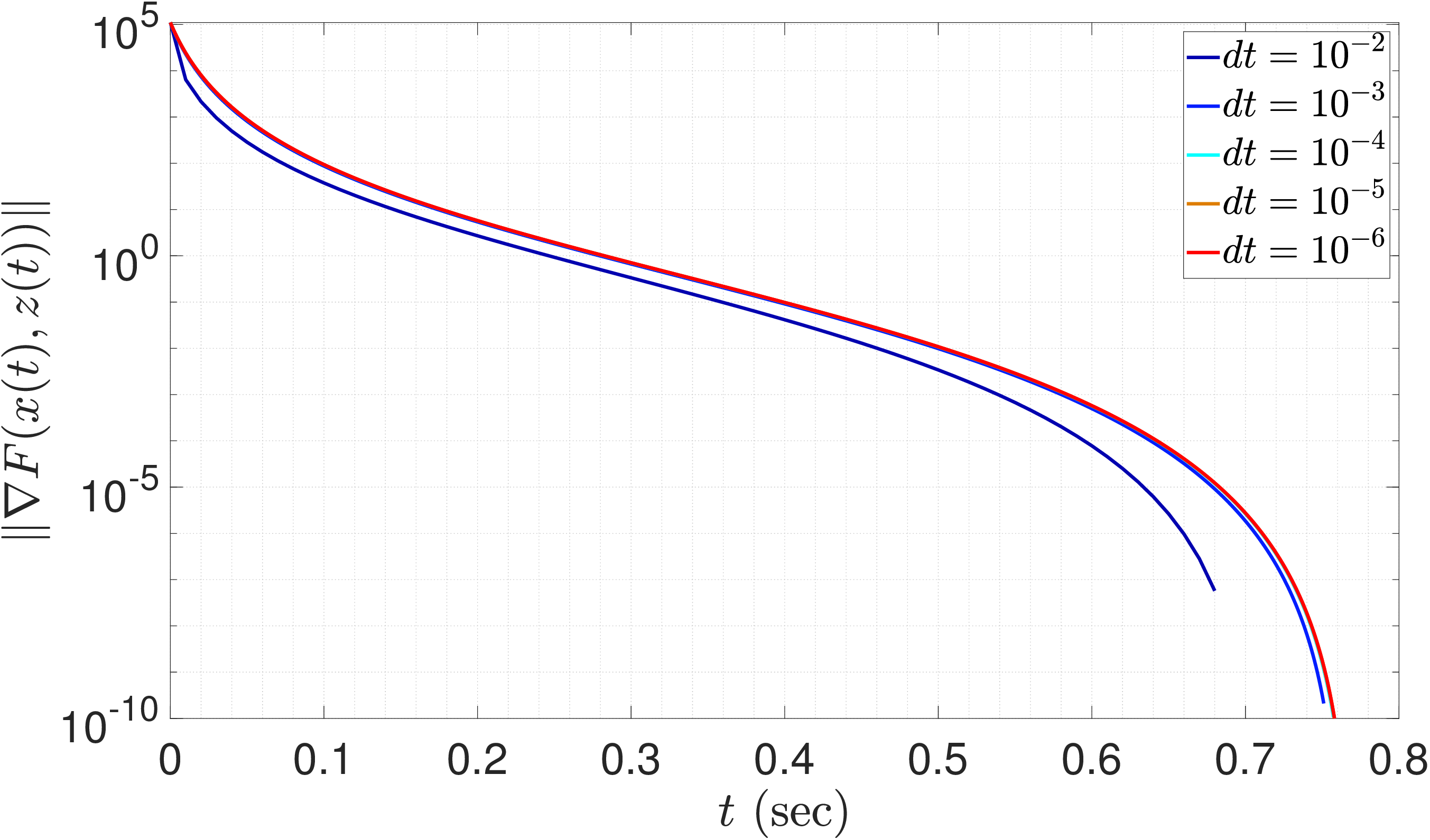}
        \caption{The norm of the gradient, $\|\nabla F(x(t),z(t))\|$, with time for various $p_1, p_2$.}
    \label{fig:nabla f dt}
\end{figure}

Finally, the performance of the proposed method is compared with the performance of the rescaled gradient flow \eqref{p flow}.
More specifically, the considered rescaled-gradient flow scheme is
\begin{equation}\label{acc SP}
    \begin{bmatrix}\dot x \\ \dot z\end{bmatrix} = -c_1\left(\nabla^2 F(x,z)\right)^{-1}\frac{\nabla F(x,z)}{\|\nabla F(x,z)\|^{\frac{p_1-2}{p_1-1}}}.
\end{equation}
where $0<\theta<1$. 
{Since the objective function in \eqref{eq: min max sim example} is only strictly convex-concave and not strongly convex-concave, \eqref{SP first order} cannot be used, but \eqref{SP dyn} can be used. The dynamical system \eqref{acc SP} is a Newton's modification of rescaled-gradient flow \eqref{p flow} discussed in \cite{wibisono2016variational}, where Hessian is used so that \eqref{acc SP} can be used for a strictly convex-concave function.}

\begin{figure}[!ht]
    \centering
        \includegraphics[ width=1\columnwidth,clip]{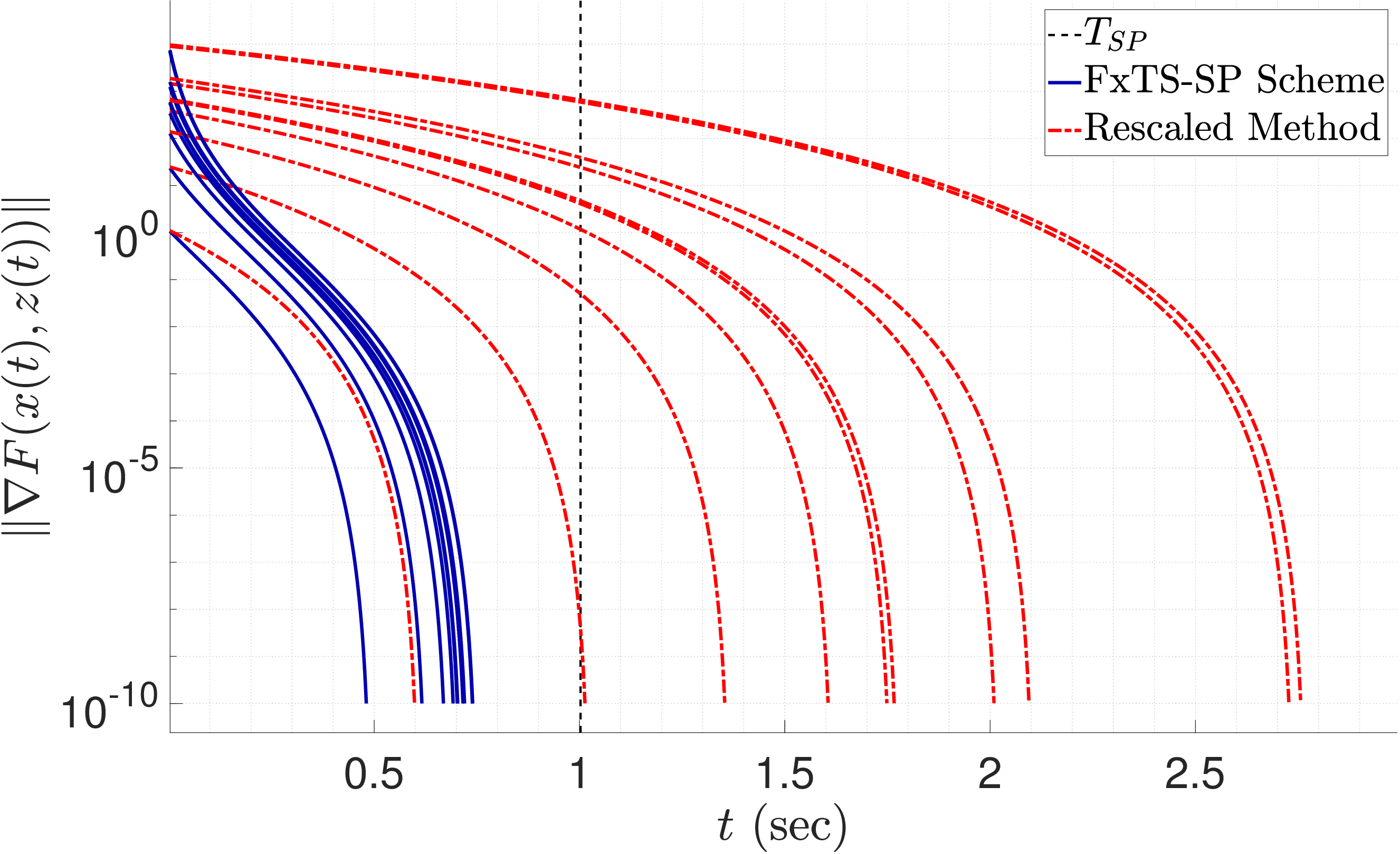}
        \caption{The norm of gradient, $\|\nabla F(x(t),z(t))\|$, with time for various initial conditions for the proposed scheme and the rescaled gradient flow scheme. 
        }\label{fig: comp sec nabx}
\end{figure}

Figure \ref{fig: comp sec nabx} plots the norm of the gradient for various initial conditions, where $p_1 = 2.2, p_2 = 1.8$, $c_1 = c_2 = 10$ for \eqref{SP dyn}, $p_1 = 2.2$, $c_1 = 10$ for \eqref{acc SP}.
It can be seen that the convergence of 
the rescaled gradient flow scheme \eqref{acc SP} is super-linear (finite-time convergence), but slower than the proposed scheme. It is evident from Figure \ref{fig: comp sec nabx} that the time of convergence for \eqref{acc SP} grows as $\|x(0)-x^\star\|$ increases, while that of the proposed scheme \eqref{SP dyn} remains bounded.

\begin{figure}[!ht]
    \centering
        \includegraphics[ width=1\columnwidth,clip]{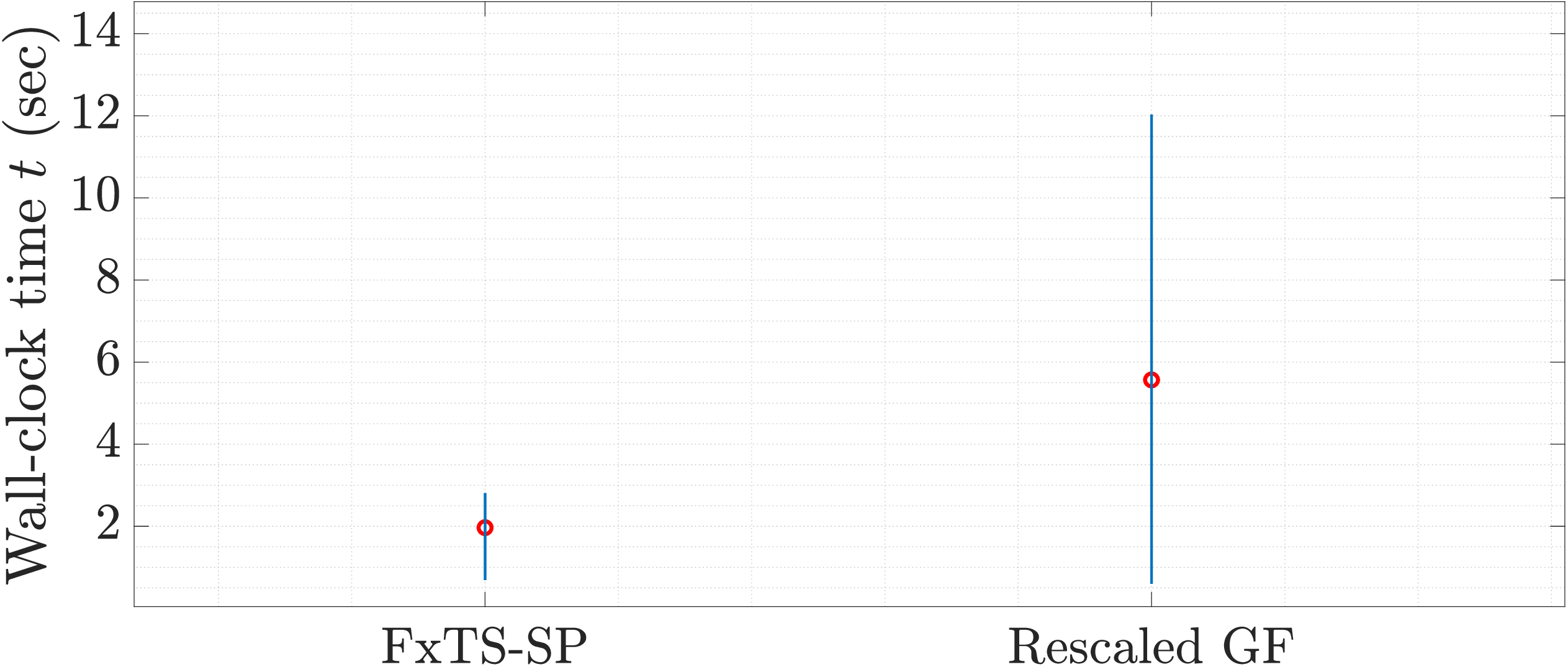}
        \caption{The wall-clock time for 1000 trials for the proposed scheme, rescaled gradient flow based scheme
        . The red dot represents the mean value for the 1000 trials while the vertical lines represent the minimum and maximum values of the respective schemes. }\label{fig: comp sec time}
\end{figure}

Figure \ref{fig: comp sec time} depicts the wall-clock time (i.e., actual run-time) for the 
two aforementioned schemes. The results are presented for 1000 trials, where the simulations are run until the norm of the gradient, $\|\nabla F(x(t),z(t))\|$, drops below $10^{-10}$. It is clear from the figures the proposed scheme takes smaller computation time than the 
accelerated scheme, while giving better convergence rate. {Note that the wall-clock time, which corresponds to the actual computational time, is different from the convergence time $T_{SP}$, which, in the discrete setting, corresponds to the number of steps required for the convergence per the relation $N = T_{SP}\times 10^5$.} It is evident that the proposed method performs better than the nominal Newton's method, as well as some of the very commonly used accelerated methods, both in terms of number of iterations required for converging to a small neighborhood of the optimal solution, and wall-clock time. 
In the next section, we discuss some of the limitations of the study of continuous-time optimization theory, and lay out directions for future work, based on recent developments in the field of rate-preserving and consistent discretization schemes.

\section{Discussion}\label{Discussions}
While optimization methods in continuous-time are important and have major theoretical relevance in general, only discrete-time algorithms are of practical use. It is an open question as to how one can discretize the dynamics \eqref{fixed flow} and other schemes presented in this work, so that the fixed-time convergence guarantees are provably preserved. While in all the numerical examples the performance of discretized implementation is at par with the theoretical results, i.e., the convergence is super-linear and the time of convergence is upper bounded by the theoretically established upper bound, the theoretical investigation on how the convergence properties are preserved after discretization is an open problem, and an active field of research (see \cite{polyakov2018consistent,polyakov2019consistent}).

In \cite{polyakov2018consistent}, the authors study a particular class of homogeneous systems, and show that there exists a \textit{consistent} discretization scheme that preserves the finite-time convergence. They extend their results to \textit{practically} FxTS systems in \cite{polyakov2019consistent}, where they show that the trajectories of the discretized system reach to an arbitrary small neighborhood of the equilibrium point in fixed time, independent of the initial condition. Given that the provided numerical examples suggest that the proposed method works efficiently even with constant-step Euler integration, the questions that naturally arise are: (i) how could the theory of consistent discretization be extended to a more general class of FTS and FxTS systems, and (ii) how this theory could be used for the methods developed in this paper so that exact convergence of iterative discrete-time optimization schemes for the proposed methods can be guaranteed in a finite or fixed number of steps. These topics are beyond the scope of the current paper, and are left open for future research.

\section{Conclusions and Future Work}\label{Conclusions}
This paper presented modified GF schemes that provide convergence of the solution to the optimal point in fixed time, under various assumptions such as strict convexity and gradient dominance, which is a relaxation of strong-convexity.
A modified version of Newton's method is also presented that possesses fixed-time convergence guarantees from any given initial condition for optimization problems with strictly convex objective function. 
Based on this result, a novel method is proposed to find the optimal point of a convex optimization problem with linear equality constraints in fixed time. 
A modified scheme for the saddle-point dynamics is proposed so that the min-max problem can be solved in fixed time. Though all the methods are presented for continuous-time optimization, numerical examples illustrate that the proposed schemes have super-linear convergence in the discretized implementation as well, that the time of convergence satisfies the theoretical bound, and that the performance of the proposed method is better than the one of commonly used algorithms, such as Newton's method, 
the rescaled gradient-based method.

Studying the general optimization problem with both equality and inequality constraints is part of the future investigations, where schemes that can converge to the optimal point in fixed time will be designed. Also, it will be of great interest to study FTS and FxTS methods of optimization on function spaces with applications such as finding the optimal barrier function for control synthesis under spatio-temporal specifications and input constraints. 
Finally, as mentioned in Section \ref{Discussions}, one of the future research directions is to investigate discretization schemes for FTS and FxTS systems that can preserve the time of convergence, and translate FTS and FxTS to convergence in finite and fixed number of steps, respectively.  
\section{Acknowledgements}
The authors acknowledge Dr. Rohit Gupta for several fruitful discussions. 

\bibliographystyle{IEEEtran}
\bibliography{myreferences}

\appendices

\section{Proof of Lemma \ref{eq point fixed point}}\label{app proof Lemma eq fixed point}
\begin{proof}
One has that $x = \bar x$ is an equilibrium of \eqref{fixed flow} if and only if {\small
\begin{equation*}
\begin{split}
    \dot {\bar x} = 0 \iff &  -c_1\frac{\nabla f(\bar x)}{\|\nabla f(\bar x)\|^\frac{p_1-2}{p_1-1}}-c_2\frac{\nabla f(\bar x)}{\|\nabla f(\bar x)\|^\frac{p_2-2}{p_2-1}} = 0\\\iff & c_1\frac{\|\nabla f(\bar x)\|}{\|\nabla f(\bar x)\|^\frac{p_1-2}{p_1-1}}+c_2\frac{\|\nabla f(\bar x)\|}{\|\nabla f(\bar x)\|^\frac{p_2-2}{p_2-1}} = 0\\
     \iff & c_1\|\nabla f(\bar x)\|^{1-\frac{p_1-2}{p_1-1}}+c_2\|\nabla f(\bar x)\|^{1-\frac{p_2-2}{p_2-1}} = 0,\\
     \iff & \|\nabla f(\bar x)\| = 0,
\end{split}
\end{equation*}}\normalsize
since $1-\frac{p_1-2}{p_1-1}, 1-\frac{p_2-2}{p_2-1}>0$ for $p_1>2$ and $1<p_2<2$. Hence, $x = \bar x$ is an equilibrium point if and only if $\nabla f(\bar x) = 0$. This completes the proof.
\end{proof}

\section{Proof of Lemma \ref{cont fixed flow}}\label{app proof Lemma cont}
\begin{proof}
Let $\mathcal X = \{x\; |\; \nabla f(x) = 0\}$. Since $f\in C^{1,1}_{loc}$, continuity of right-hand side of \eqref{fixed flow} is immediate on $\mathbb R^n\setminus\mathcal X$. Let $\bar x\in \mathcal X$  and $L$ be the Lipschitz constant for function $\nabla f$, i.e., $\|\nabla f(x)-\nabla f(y)\| \leq L\|x-y\|$ for $x,y\in D$ where $D$ is some open neighborhood of $\bar x$. For $y = \bar x$, it follows that $\|\nabla f(x) - \nabla f(\bar x)\| = \|\nabla f(x)\|$; then, using continuity of the norm, one has {\small
\begin{equation*}
\begin{split}
     \left\|\lim_{x\rightarrow\bar x}c_1\frac{\nabla f(x)}{\|\nabla f(x)\|^\frac{p_1-2}{p_1-1}}\right\| & = \lim_{x\rightarrow\bar x}c_1\left\|\frac{\nabla f(x)}{\|\nabla f(x)\|^\frac{p_1-2}{p_1-1}}\right\|\\ 
     & = c_1\lim_{x\rightarrow\bar x} \|\nabla f(x)\|^{1-\frac{p_1-2}{p_1-1}}  \\
     & = c_1 \lim_{x\rightarrow\bar x}\|\nabla f(x)\|^{\delta_1}\\
     & \leq c_1L^{\delta_1}\lim_{x\rightarrow\bar x} \|x-\bar x\|^{\delta_1} = 0,
\end{split}
\end{equation*}}\normalsize
where $\delta_1 = 1-\frac{p_1-2}{p_1-1}>0$ for $p_1>2$. Hence, one has that $ \lim_{x\rightarrow\bar x}c_1\frac{\nabla f(x)}{\|\nabla f(x)\|^\frac{p_1-2}{p_1-1}} = 0$. Similarly, it can be shown that $\lim_{x\rightarrow\bar x}c_2\frac{\nabla f}{\|\nabla f\|^\frac{p_2-2}{p_2-1}}=0,$ since $\delta_2 = 1-\frac{p_2-2}{p_2-1}>0$ for all $1<p_2<2$. Per Lemma \ref{eq point fixed point}, one has that $\bar x$ is an equilibrium of \eqref{fixed flow}. This implies that the right-hand side of \eqref{fixed flow} is continuous at $x = \bar x$, for all $\bar x\in \mathcal X$, and hence, is continuous for all $x\in \mathbb R^n$. 
\end{proof}

\section{Proof of Lemma \ref{smooth convex}}

\begin{proof}\label{app smooth convex}
The {convexity and strong-smoothness} assumptions on $f$ implies that $f^{**} = f$, i.e., $f$ is the conjugate of its conjugate $f^*$. Define $\kappa = f^*$ so that one has $\kappa^* = f^{**} = f$. Now, since the function $f$ is the conjugate of $\kappa$ and is $\beta$-strongly smooth, from \cite[Section 3.5]{zalinescu2002convex}, one obtains that there exists $\beta^*$ such that $\kappa$ is a $\beta^*$-strongly convex function. It holds that if $A$ is full row-rank, then $\beta^*$-strong-convexity of $f^*$ implies $\alpha$-strong-convexity of $f^*(-A^T\nu)$, where $\alpha  = \lambda \beta^*$ and $\lambda = \lambda_{min}(AA^T)$ is the minimum eigenvalue of $AA^T$. Since $A$ is full row-rank, it follows that $\lambda>0$. Finally, using the fact that $f_1 = f^*(-A^T\nu)$ is $\alpha$-strongly convex and $f_2 = \nu^Tb$ is convex, one obtains that $f_1+f_2 = f^*(-A^T\nu) + \nu^Tb = -g(\nu)$ is $\alpha$-strongly convex, or equivalently, $g$ is $\alpha$-strongly concave. 
\end{proof}

\section{Proof of Lemma \ref{M inv}}\label{app proof M inv}
\begin{proof}
Define $H_{xx} = \nabla_{xx} F$, $H_{xz} = \nabla_{xz} F$ and $H_{zz} = -\nabla_{zz} F$. Since $F$ is twice-continuously differentiable, one has that $\nabla_{zx} F = (\nabla_{xz} F)^T$. Define $H =\nabla^2 F(x,z)$ so that $H = \begin{bmatrix} H_{xx} & H_{xz}\\ H_{xz}^T & -H_{zz}\end{bmatrix}$. Note that $H_{xx}$ and $H_{zz}$ are positive definite for all $(x,z)\in U$ due to Assumption \ref{F assum 1}. The rank of the matrix $H$ satisfies (\cite{paige2008hua}) 
\begin{equation*}
\begin{split}
    \rank H & = \rank H_{xx} + \rank(-H_{zz} - H_{xz}^TH_{xx}^{-1}H_{xz}).
\end{split}
\end{equation*}
Now, since $H_{xx}$ is invertible for all $(x,z)\in U$, one has that $\rank H_{xx} = n$. Let $H_1 = H_{xx}$ and $H_2 =-H_{zz} - H_{xz}^TH_{xx}^{-1}H_{xz}$. Since $H_{xx}, H_{zz}$ are positive definite matrices, it follows that $H_2$ is also negative definite. Hence, one obtains that $\rank H_2 = m$. This implies that $\rank H = \rank H_1 + \rank H_2 = n+m$ for all $(x,z)\in U$, i.e., $\nabla^2 F(x,z)$ is full rank and hence, invertible for all $(x,z)\in U$. 
\end{proof}

\end{document}